\documentclass{amsart}

\usepackage[foot]{amsaddr}

\usepackage[margin=0.8in]{geometry}

\title[OWL$_{2,1}$ regularization for rank-aware recovery]{%
Orthogonally weighted $\ell_{2,1}$  regularization for rank-aware joint sparse recovery: algorithm and analysis
}

\usepackage{hyperref}

\usepackage[utf8]{inputenc} 
\usepackage[T1]{fontenc}    
\usepackage{amsthm}
\usepackage{ dsfont }
\usepackage{amssymb}
\usepackage{fixmath}
\usepackage{enumerate}
\usepackage{graphicx}
\usepackage{float}
\usepackage{bbm}
\usepackage{amsmath}
\usepackage{mathtools}
\usepackage{comment}
\usepackage{hyperref}
\usepackage{listings}
\usepackage{color}
\usepackage{bm}
\usepackage[dvipsnames]{xcolor}
\usepackage{algorithm} 
\usepackage{algpseudocode} 

\hypersetup{
    colorlinks,
    citecolor=blue,
    filecolor=blue,
    linkcolor=blue,
    urlcolor=blue
}

\newtheorem{theorem}{Theorem}[section]

\newtheorem{proposition}[theorem]{Proposition}
\newtheorem{lemma}[theorem]{Lemma}

\newtheorem{corollary}[theorem]{Corollary}
\newtheorem{definition}[theorem]{Definition}
\newtheorem{remark}[theorem]{Remark}

\newtheorem{example}{Example}

\newcommand{\R}{\mathbb{R}}

\newcommand{\prox}{\operatorname{Prox}}

\DeclarePairedDelimiter{\norm}{\lVert}{\rVert}
\DeclarePairedDelimiter{\abs}{\lvert}{\rvert}

\DeclareMathOperator{\tr}{tr}
\newcommand{\St}{\mathbb{S}}

\newcommand{\fro}{\textnormal{Fro}}
\newcommand{\op}{\textnormal{op}}

\newcommand{\rank}{{\rm rank\,}}

\def\spark{\textnormal{spark\,}}

\newcommand{\owl}{\mathit{ow}\ell_{2,1}}
\newcommand{\ellone}{\ell_{2,1}}
\newcommand{\ellzero}{\ell_{2,0}}

\definecolor{darkgreen}{rgb}{0.0, 0.5, 0.0}
\definecolor{orange}{rgb}{0.8, 0.33, 0.0}

\usepackage[obeyFinal]{todonotes}

\DeclareMathOperator*{\argmin}{argmin}			

\begin{document}

\date{\today}
\author{Armenak Petrosyan} 
\address{
  School of Mathematics, Georgia Institute of Technology, 686 Cherry St NW, Atlanta, GA 30332, USA.
}
\email{apetrosyan3 (at) gatech.edu}
\author{Konstantin Pieper}
\email{pieperk (at) ornl.gov}
\author{Hoang Tran}
\address{Computer Science and Mathematics Division, Oak Ridge National Laboratory, Oak Ridge, TN 37831, USA.}
\email{tranha (at) ornl.gov}

\thanks{
Armenak Petrosyan's research was partially funded by a Georgia Tech-ORNL SEED grant. Armenak Petrosyan also thanks Christopher Heil for continued support during his employment at Georgia Institute of Technology. Konstantin Pieper acknowledges support through the Compression Methods for Streaming Scientific Data project. Hoang Tran acknowledges support from the Scientific Discovery through Advanced Computing (SciDAC) program through the FASTMath Institute under Contract No.\ DE-AC02-05CH11231. This manuscript has been authored by UT-Battelle, LLC, under Contract DE-AC05-00OR22725 with the US Department of Energy (DOE). The US government retains and the publisher, by accepting the article for publication, acknowledges that the US government retains a nonexclusive, paid-up, irrevocable, worldwide license to publish or reproduce the published form of this manuscript, or allow others to do so, for US government purposes. DOE will provide public access to these results of federally sponsored research in accordance with the DOE Public Access Plan (\url{http://energy.gov/downloads/doe-public-access-plan}). 
}

\maketitle

\begin{abstract}
 We propose and analyze an efficient algorithm for solving the joint sparse recovery problem using a new regularization-based method, named orthogonally weighted $\ell_{2,1}$ ($\mathit{ow}\ell_{2,1}$), which is specifically designed to take into account the rank of the solution matrix. This method has applications in feature extraction, matrix column selection, and dictionary learning, and it is distinct from commonly used $\ell_{2,1}$ regularization and other existing regularization-based approaches because it can exploit the full rank of the row-sparse solution matrix, a key feature in many applications. We provide a proof of the method's rank-awareness, establish the existence of solutions to the proposed optimization problem, and develop an efficient algorithm for solving it, whose convergence is analyzed. We also present numerical experiments to illustrate the theory and demonstrate the effectiveness of our method on real-life problems.

\end{abstract}
\section{Introduction}
\label{sec:intro}

The joint sparse recovery or multiple measurement vector (MMV) problem aims to find a solution to the matrix equation $AX = Y$, where $A\in \R^{M\times N}$ and $ Y\in \R^{M\times K}$ are given, and $X \in \R^{N \times K}$ is an unknown $s$-row sparse matrix, meaning that at most $s$ rows of $X$ are non-zero.  

The task of finding the solution of $AX=Y$ with the fewest non-zero rows in the noiseless case can be formulated as an optimization problem using mixed $\ell_{2,0}$ norms:
\begin{equation}\label{prob:ell20}
\min_{Z\in \R^{N\times K}}\ellzero(Z),\ \ \text{subject to}\ \ AZ=Y,
\end{equation}
where $\ellzero(Z)$ is the number of non-zero rows in $Z$. However, the main drawback of this method is that the optimization problem is computationally infeasible, as it results in an exhaustive search that is an NP-hard problem. To overcome this limitation, the $\ell_{2,1}$ norm regularization method is commonly used to solve the joint sparse recovery problem \cite{doi:10.1137/100785028} 
\begin{equation}\label{prob:l21}
\min_{Z\in \R^{N\times K}}\ellone(Z),\ \ \text{subject to}\ \ AZ=Y,
\end{equation}
where $\ellone(Z) =\norm{Z}_{2,1} = \sum_{n=1}^N\sqrt{z_nz_n^T}$,
and $z_n$ is the $n$-th row of $Z$. This method generalizes the $\ell_1$ sparsifying penalty to the multiple vector setting and is computationally feasible. However, this method has the drawback of being rank-blind, meaning it does not take into account the rank of the $s$-row sparse matrix \cite{davies2012rank}. In real-world scenarios, data typically has a maximum essential rank, equal to its essential sparsity. 

In \cite{petrosyan2019reconstruction}, it was argued that a rank-aware regularizer cannot be continuous and, thus, convex in all of $\R^{N\times K}$. A new  method for recovery of jointly sparse vectors in the form of a (non-convex) optimization problem was proposed 
\begin{equation}\label{prob:nonconvex}
    \min_{Z\in \R^{N\times K}}\owl(Z)\text{  subject to  } AZ=Y,
\end{equation}
where the regularizer, here we call the orthogonally weighted $\ell_{2,1}$ or $\mathit{ow}\ell_{2,1}$, is defined as 
\[
  \owl(Z)= \ellone\left(Z(Z^TZ)^{\dagger/2}\right)
  = \sum_{n=1}^N \sqrt{z_n(Z^TZ)^\dagger z_n^T},
\]
with \((\cdot)^\dagger\) denoting the Moore-Penrose pseudo-inverse and \((\cdot)^{\dagger/2}= ((\cdot)^\dagger)^{1/2}\) its square root for a given symmetric matrix. 
If the minimization is restricted to matrices of full rank, we obtain the functional
\(\owl(Z) = \ellone(Z(Z^TZ)^{-1/2})\). The papers
 \cite{mishali2008reduce, petrosyan2019reconstruction} also discussed a strategy to reduce the joint sparse problem~\eqref{prob:ell20} to a simpler setting where the measurement matrix has full  column-rank. 
This transformation can be achieved through techniques such as Singular Value Decomposition (SVD) or other factorization methods.
In particular, if $Y = Y'Q$ where $Y' \in \mathbb{R}^{M\times r}$ with $\rank(Y') = r$ and $Q \in \R^{r\times K}$ with $QQ^T = I \in \R^{r\times r}$, then for any solution $Z'$ of the optimization problem
\[
  \min_{Z \in \R^{K\times r},\, \rank(Z) = r} \;\ellzero(Z)\text{  subject to  } AZ = Y',
\]
the matrix $\bar{Z} = Z' Q$ is a solution of original problem~\eqref{prob:ell20}.
This strategy is independent of the chosen regularization functional and can be used for the noisy problem as well to reduce it to an approximate full column rank problem, and the resulting reconstruction error can be controlled; see \cite{petrosyan2019reconstruction} for more details.

However, \cite{petrosyan2019reconstruction} neither provided an affirmative conclusion on the rank-awareness, nor provided a general algorithm for implementation of $\owl$ regularization.
The algorithm proposed in~\cite{petrosyan2019reconstruction} relies on being able to determine the rank of the solution directly from the data and does not directly address the more practical case of noisy data discussed below.
In this paper, we aim to address these two important open questions by providing: i) a proof of the rank-awareness of the $\owl$ regularization, making it the first regularization-based method possessing this property; and ii) a computationally efficient algorithm for solving the optimization problem, supported by rigorous convergence analysis.





The naming of the functional is derived from the following. Note that, if $Z=U\Sigma V^T$ is the compact singular value decomposition of $Z$,
i.e., with \(U \in \R^{N\times r}\), \(V\in \R^{K\times r}\) orthogonal, $r= \rank(Z)$, and \(\Sigma \in \R^{r\times r}\) diagonal and s.p.d.,
then $Z(Z^TZ)^{\dagger/2} = UV^T$, thus it renormalizes $Z$  to its orthogonal component before applying the $\ell_{2,1}$ norm.
Also
\[
  \owl(Z)=\ellone(UV^T)=\ellone({U}),
\]
in other words, $\owl(Z)$ is the $\ell_{2,1}$ norm of an orthonormal basis of the column space of $Z$ and its value is independent of the choice of the basis. This and many other properties of the $\owl(Z)$ functional, like its locally Lipschitz continuity, are discussed in detail in Appendix~\ref{sec:rankaware}.
Consequently, when $Z$ is of $s$-row sparsity and of rank $s$, it can be seen that $\owl(Z)=\ellzero(Z)$, in other words it interpolates the values of $\ell_{2,0}$ on such matrices. However, unlike the latter, it is feasible to develop an efficient algorithm for solving the optimization problem with the $\mathit{ow}\ell_{2,1}$ regularizer, as we shall see in this effort. 

In the presence of measurement noise,  we consider the regularized formulation of~\eqref{prob:nonconvex}
\begin{equation}
\label{prob:nonconvex_reg}
\min_{Z\in \R^{N\times K} } \left(\owl(Z)+\frac{1}{2\alpha} \|AZ-Y\|_\fro^2\right).
\end{equation}
The term $1/(2\alpha) \|AZ-Y\|_\fro^2$ acts as a penalty for deviating from the measurements, where $\|\cdot\|_\fro$ denotes the Frobenius norm of a matrix. The quantity \(1/(2\alpha)\) controls the strength of the regularization or can be interpreted as a Lagrange multiplier for the constrained formulation
\begin{equation}
  \label{prob:nonconvex_constr}
\min_{Z\in \R^{N\times K} } \owl(Z)
\quad\text{subject to } \|AZ-Y\|_\fro \leq \delta.
\end{equation}
To simplify notation, we will primarily consider the regularized formulation of the problem as given in~\eqref{prob:nonconvex_reg}.

\subsection{Notational conventions }\label{sec:notation} 
Here we introduce various notation conventions aimed at enhancing the clarity and readability of the content. The true value of the matrix of interest is represented by $X \in \mathbb{R}^{N \times K}$. The optimal solution of a given regularization problem is represented by $\bar Z \in \mathbb{R}^{N \times K}$.
The rows of a given matrix $Z$ are denoted as $z_n \in \mathbb{R}^K$ for $n = 1,\dots,N$.
For a given symmetric positive matrix \(W\), the $W$-(semi-)norm of a vector $z \in \mathbb{R}^K$ is represented by $\|z\|_W = \sqrt{z^T W z} = \norm{W^{1/2}z}_2$, and the Euclidean norm is denoted by \(\|z\|_2\).
The $\ell_{W,p}$ norm for \(p\in [1,\infty)\) of $Z$ is represented by
\[
  \ell_{W,p}(Z)=\norm{Z}_{W,p}= \left(\sum_{n=1}^N\norm{z_n}^p_W\right)^{1/p}
\]
and for \(p= \infty\) the sum is replaced by a maximum, as usual.
Note that for \(p=2\) we obtain \(\norm{Z}_{W, 2} = \sqrt{\tr(ZWZ^T)}\) and the special case for the Euclidean norm is also denoted as the Frobenius norm \(\norm{\cdot}_{2,2} = \norm{\cdot}_{\fro}\).
The associated inner product on the space of matrices is denoted by \(\langle Z, Z'\rangle = \tr(Z^T Z')\).
Finally, by \(\norm{\cdot}_{\op}\) we denote the operator or induced matrix norm with respect to the Euclidean inner product.
For a given matrix \(M\), we denote by \(M^\dagger\) its Moore-Penrose pseudo-inverse, and for a given symmetric non-negative matrix \(M\) by \(M^{\dagger/2} = (M^\dagger)^{1/2}\) the square root of the pseudo-inverse.
The main functional of interest is represented by $\owl(Z) = \ellone(Z(Z^TZ)^{\dagger/2}) = \norm{Z}_{(Z^TZ)^\dagger,1}$.
  Note that
  $U = Z(Z^TZ)^{\dagger/2}$ is related  to the matrix sign function; see~\cite[Chapter~5]{higham2008functions} with the difference being that we use singular value calculus, as opposed to the standard (eigenvalue based) matrix calculus.
  In case of scalars or symmetric matrices the two definitions are identical. 

\subsection{Organization} 
The material is organized as follows: Section~\ref{sec:related} discusses the applications of joint sparse recovery and related works. Section~\ref{sec:rank_aware} proves the rank-awareness of $\owl$ regularization.
Section~\ref{sec:glob_min_noiseless} establishes conditions for the existence of global minimizers of~\eqref{prob:nonconvex} and~\eqref{prob:nonconvex_reg}.
The question of how well the solution of the regularized problem approximates the true value remains open.
In Section~\ref{sec:approx}, we provide a family of regularizers for an additional parameter \(\gamma\) that connects the $\owl$ regularizer at \(\gamma = 0\) and the convex \(\ellone\) regularizer at \(\gamma = 1\) and establish first-order conditions for the generalized problem.
This includes optimality conditions for the \(\owl\) regularized problem in the situation where \(Z^TZ\) has full rank. For other cases, we provide a convergence result for \(\gamma\to 0\), and rely on the optimality conditions for the ``relaxed problem'' for \(\gamma > 0\).
Section~\ref{sec:algorithm} proposes and analyzes a numerical algorithm for solving the optimization problem that is based either on a direct minimization of~\eqref{prob:nonconvex_reg} with a variable metric proximal gradient method or on a combination of this with a continuation strategy in \(\gamma\).
Section~\ref{sec:experiments} conducts experiments to demonstrate the performance of the algorithm.
All proofs concerning various necessary properties of the regularizer $\owl(Z)$ have been moved to the appendix to improve readability.

\section{Literature review}\label{sec:related}

\hspace{0.15in} \textbf{$\ell_1/\ell_2$ regularization.} Our method represents a generalization of $\ell_1/\ell_2$, which uses the ratio of $\ell_1$ and $\ell_2$ to model the sparsity \cite{hoyer2004non,doi:10.1137/13090540X}.
Indeed, in the special case of single vector recovery (i.e.\ for $K=1$), the problems~\eqref{prob:nonconvex} and~\eqref{prob:nonconvex_reg} are equivalent to the $\ell_1/\ell_2$ regularized problem due to
\begin{equation}
  \label{eq:l1overl2}
  \owl(Z)
  = \norm{Z(Z^T Z)^{\dagger/2}}_1
  = \begin{cases}
    \norm{Z}_1 / \norm{Z}_2 & \text{for } Z\neq 0, \\
    0 & \text{for } Z = 0,
  \end{cases}
  \qquad \text{if } Z \in \R^{N\times 1}.
\end{equation}
Theoretical analysis for $\ell_1/\ell_2$ has been developed in \cite{Yin2014RatioAD} for non-negative signals, \cite{rahimi2019scale} with a local optimality condition, and \cite{doi:10.1137/20M136801X}, where an analytic solution for the proximal operator is characterized, among others. Our proof of the existence of the global solution to~\eqref{prob:nonconvex} and~\eqref{prob:nonconvex_reg} here is built on a generalization of the \textit{spherical section property}, which has been studied in \cite{doi:10.1137/20M1355380,XU2021486} for the global optimality condition of $\ell_1/\ell_2$. The $\ell_1/\ell_2$ functional has been shown to enhance sparsity of the solutions and outperform $\ell_1$ on many test problems such as nonnegative matrix factorization and blind deconvolution \cite{5995521,JI2012295}.  However, the existing uniform recovery guarantees for $\ell_1/\ell_2$ remains too strict in the general, noisy setting (i.e., more restrictive than well-known $\ell_1$ recovery conditions such as the null space property), despite many recent theoretical advances.

\textbf{Joint sparse recovery.} The joint sparse recovery problem, also known as multiple measurement vectors, distributed compressed sensing or simultaneous sparse approximation, has garnered a lot of attention in the literature. 
Mathematically, this problem arises when we want to reconstruct multiple vectors which have the same sparsity support  using the same set of linear combinations for each. While one can certainly recover each vector independently, approximating them all at once should give superior results because we can take advantage of additional information about their shared support. To realize this potential, however, a joint sparse recovery algorithm is desired to be rank aware, that is, its performance improves with increasing rank of the solution matrix. 

Existing approaches for the joint sparse recovery can be roughly categorized into greedy methods \cite{1453780,TROPP2006572,GRSV08}, and algorithms based on mixed norm optimization \cite{1453780,TROPP2006589,FornasierRauhut08,EldarRauhut10,5452189}. Pioneering joint sparse methods are often straightforward extensions of algorithms for single vector recovery, whose recovery conditions are based on standard compressed sensing properties and typically turn out to be the same as those obtained from a single measurement problem. In such cases, the theoretical results do not show any advantage of the joint sparse approach. One of the first methods exploiting rank information is a discrete version of the MUSIC (MUltiple SIgnal Classification) algorithm \cite{544131}, where recovery guarantees better than those from single vector problem are proved when the solution matrix has full column-rank. Extensions of this work to non-full-rank cases have been developed in \cite{6158602} (subspace-augmented MUSIC), \cite{6122004} (compressive MUSIC) and \cite{7875091} (semi-supervised MUSIC). Also inspired by the MUSIC technique, \cite{davies2012rank} proposes the rank-aware orthogonal recursive matching pursuit (RA-ORMP) based on a greedy selection strategy, and further demonstrates the rank awareness and rank blindness for a large class of joint sparse recovery algorithms. That work proves that some of the most popular algorithms, including orthogonal matching pursuit and mixed $\ell_{q,1}$ norm for any $q\ge 1$, are effectively rank blind. More recent rank aware methods include Simultaneous Normalsize Iterative Hard Thresholding (SNIHT) \cite{6719509} and Rank Aware Row Thresholding (RART) \cite{blanchard2020rank}. Compared to the greedy approach, the integration of rank information into regularization-based algorithms seems an open problem. To date, all existing rank-aware algorithms belong to the greedy class.

\textbf{Applications.} The joint sparse recovery problem finds applications in various fields, including data analytics and imaging \cite{10.5555/1795114.1795154,6288484}, uncertainty quantification \cite{refId0}, and parameter identification of dynamical systems \cite{Schaeffer2017LearningDS}. Here, we will illustrate one particular area of our interest in detail, which is the Column Subset Selection Problem (CSSP), also known as the feature selection, or the dictionary selection problem. Suppose we have a set of data columns $Y\in \R^{M\times K}$ and a dictionary of features $A\in \R^{M\times N}$ (here $N$ represents the number of features in the dictionary and $M$ is the size of the features), and we want to extract a subset of the columns of $A$ such that $Y$ can be represented in terms of only $s$ columns of $A$. This problem exactly translates to the joint sparse recovery problem $Y=AX$ where $X$ is the matrix of coefficients and the non-zero rows correspond to the columns that have been selected. In the feature selection problem, the data is also the matrix $A$ (i.e., $Y=A$) and the goal is to select the most representative columns (features) of the matrix $A$. Thus, we aim to solve the problem
\begin{align}
\label{problem:CSSP}
  \min_Z\|A-AZ\|_\fro\;\text{subject to}\; \|Z\|_{2,0}\leq s.
\end{align}
This problem has garnered extensive attention in machine learning where it was often relaxed by mixed $\ell_{2,1}$ norm~\cite{10.5555/1795114.1795154,nie2010efficient,lu2019embedded,nardone_SMBA}. This has also led to various matrix decomposition techniques such as the CUR decomposition~\cite{hamm2020perspectives, cai2021robust}. Demonstrations of our $\owl$ regularization on this problem on some bioinformatics datasets will be shown in Section~\ref{sec:experiments}.

\textbf{Proximal algorithms.}
To our knowledge, this is the first paper that provides an iterative solution method for the problem~\eqref{prob:nonconvex_reg} aside from the initial work~\cite{petrosyan2019reconstruction}.
The methods we outline in Section~\ref{sec:algorithm} will be based on standard proximal gradient based methods for least squares regression with convex regularizers~\cite{doi:10.1137/080716542,chambolle-pock-2011}, and for more general non-smooth, convex optimization that encompasses joint sparse regression~\cite{doi:10.1080/02331930412331327157,FornasierRauhut08,doi:10.1137/060669498,tan2014joint,Dexter-SVVA22}.
In order to incorporate the nonconvex non-smooth aspects of the \(\owl\)-regularizer, we employ a first order convex approximation through a model function in every step of the method.
This allows us to write the new iteration with the help of a proximal mapping that is simple to calculate.
Convex model functions are also used in related methods for nonsmooth nonconvex optimization~\cite{NollProtRondepierre:2008,Noll:2010} or Trust-Region methods~\cite{ConnGouldToint:2000}.
Our method employs a specialized variable metric to achieve simple and efficient iterations, which requires a dedicated analysis and leads to a global convergence result.

As already pointed out before, for the scalar case \(K=1\) we obtain the \(\ell_1/\ell_2\) regularizer, and optimization methods have been proposed in~\cite{Yin2014RatioAD,rahimi2019scale,9057443,XU2021486,doi:10.1137/20M136801X,doi:10.1137/20M1355380}.
For \(K=1\), proximal method in this manuscript is similar to the one proposed in~\cite[Algorithm~2]{9057443} for the equality constrained formulation~\eqref{prob:nonconvex} and subsequently analyzed and generalized to include the noisy constrained formulation~\eqref{prob:nonconvex_constr} in~\cite{doi:10.1137/20M1355380}.
In this paper we focus on the noisy regularized formulation~\eqref{prob:nonconvex_reg}, which can also be used to approach the solutions of the equality constrained problem for \(\alpha \to 0\) and the solutions of the noisy constrained problem for \(\alpha \to \alpha_\delta\); see Section~\ref{sec:adapt_alpha}.
While it is easy to specialize the method proposed in this manuscript to \(K=1\) (see Remark~\ref{rem:algorithm_for_K1}), the generalization to larger \(K\) is not obvious, and requires additional effort.

Finally, since the \(\owl\) functional is discontinuous across different ranks, we can not rely on only gradient descent unless we know or can guess the optimal rank ahead of time~\cite{petrosyan2019reconstruction}.
As a remedy, we propose an optional Lipschitz-continuous relaxation of the problem for a parameter \(\gamma>0\), which we show converges to the original problem in the sense of \(\Gamma\)-convergence~\cite{dal2012introduction}.

\section{Rank-awareness}\label{sec:rank_aware}

Denote for any \(r\leq K\)
$$
\St(N,K,r)  = \{Z\in \R^{N\times K}: \rank(Z)=r\}.$$
When  the columns of $Z$ are linearly independent ($r=K$), we abuse the notation and write $\St(N,r) = \St(N,K,r)$. It is worth noting that the set $\St(N,K,r)$ is often referred to as the Stiefel manifold and it is a key object in differential geometry and optimization, with numerous applications in machine learning and data analysis. 

The following theorem \cite{lai2011null, davies2012rank} provides a necessary and sufficient condition for the exact recovery of sparse matrices from the $\ellone$ minimization problem~\eqref{prob:l21}:
\begin{theorem}\label{nspcond}
        Given a matrix $A\in \R^{M\times N}$, every  s-row sparse matrix $X\in \R^{N\times K}$  can be exactly recovered from~\eqref{prob:l21} if and only if $A$ satisfies the NSP: 
        $$
        \|x_I\|_{1}<\|x_{I^c}\|_{1},\ \forall x\in \ker(A)\setminus \{\mathbf{0}\},\ \forall I\subset\{1,\dots,N\}, \text{ with } |I|\leq s,
        $$ 
        where $ x_{I} $ is the vector that we get after nullifying all the entries of $ x $ except the ones with indexes in $ I $. Furthermore, if there exists an $ x\in \ker(A)$ such that $$\|x_I\|_1>\|x_{I^c}\|_1,\; |I|\leq s,$$
     then, for any $1\leq r\leq \min\{s,K\}$, there is an $s$-row sparse matrix $X\in \St(N,K,r)$  which cannot be recovered from~\eqref{prob:l21}. 
\end{theorem}

The following theorem (see \cite{petrosyan2019reconstruction} for details) highlights the difference in behavior between the $\ell_{2,1}$ norm and the $\ell_{2,0}$ norm regularizers. The $\ell_{2,1}$ norm regularizer is not rank-aware, while the $\ell_{2,0}$ norm regularizer is rank-aware. The smallest number of linearly dependent columns in a matrix $A$ is referred to as the spark of $A$ and is denoted by $\spark(A)$.

\begin{theorem}\label{thm:rankrecequiv}
The following statements are equivalent:
\begin{enumerate}
\item\label{cond:rankrecequiv_1}  $ s\leq (\spark(A)-1+r)/2 $, 
\item For every $s$-row sparse matrix $ X\in \St(N,K,r)$, $ X $ is the unique solution of $AZ=Y$ in the set $\St(N,K,r)$.
\item For every $s$-row sparse matrix $ X\in \St(N,K,r)$, $ X $ is the unique solution of the minimization problem~\eqref{prob:ell20}
where $ Y=AX $.
\end{enumerate}
\end{theorem}



In the above theorem, the least restrictive condition occurs when the rank of the row-sparse matrix $X$ is equal to its sparsity, i.e., $r = s$. In fact, in practical applications, the row-sparse matrix $X$ in the joint sparse recovery problem is often of maximum rank, i.e., it is approximately $s$-row and also $s$-row sparse. In such a scenario,  the condition specified in Theorem~\ref{thm:rankrecequiv} simplifies to $s < \spark(A)$.    Hence, the assumption $s<\spark(A)$ is consistently made throughout the manuscript.

\begin{theorem} \label{thm:unique_is_truth} Given a matrix $A\in \R^{M\times N}$ with the property $s < \spark(A)$, let $X\in \St(N,K,s)$ represent an $s$-row sparse matrix, and define $Y = AX$.  Then $X$ is the unique solution to~\eqref{prob:nonconvex}. 
\end{theorem}
\begin{proof}From Prop.~1 in \cite{petrosyan2019reconstruction}, since $s<\spark(A)$, we have $\rank(AX)=\rank(X)$. Let $Z\in \R^{N\times K}$ be such that $AX=AZ$, then $\rank(Z)\geq \rank(AZ)=\rank(X)$. Combining with Lemma~\ref{lem:Psibound}, $\ellzero(Z)\ge \ellzero(X)$, with equality holding if and only if $Z$ is $s$-row sparse.
\end{proof}

\begin{corollary} The $\owl$ regularization in~\eqref{prob:nonconvex} is rank aware. 
\end{corollary}
\begin{proof}
According to Theorem~\ref{thm:rankrecequiv}, not all $s$-row sparse matrices $X \in \St(N, K, 1)$ can be uniquely recovered from the equation $AZ = Y$ where $Y = AX$, when $\spark({A}) / 2 < s < \spark({A})$ and $\spark(A) > 2$. However, as per Theorem~\ref{thm:unique_is_truth}, all $s$-row sparse matrices $X \in \St(N, K, s)$ can be uniquely recovered using the method described in~\eqref{prob:nonconvex}.
\end{proof}

For a given matrix $A$, if ${\spark(A)}/{2} < s < \spark(A)$ and $r<s$, it is not necessarily guaranteed that the solution to the joint sparse recovery problem exists or is unique. In this context, Section~\ref{sec:glob_min_noiseless} introduces a condition named spherical section property to ensure the existence of the global solution.

Earlier, we discussed that when the rank is equal to one, the $\owl$ reduces to $\ell_1/\ell_2$.
Below we consider examples in which the $\ell_1/\ell_2$ fails to recover a single vector but if we add another vector along with the corresponding measurements, then we successfully recover both.

\begin{example} \label{ex:ssp_fail_1} Consider a matrix $A\in \mathbb{R}^{2\times 3}$ and $X\in \mathbb{R}^{3\times 1}$, where
\begin{equation*}
A = 
\begin{pmatrix}
c & 1 & 0 \\
c & 0 & 1 
\end{pmatrix},
\quad 
X=\begin{pmatrix}
0\\
1 \\
1
\end{pmatrix}
\quad 
\text{  and  }\quad  c \geq 0.
\end{equation*}
Then the measurement matrix $Y$ is given by $Y = AX=
(1, 1)^T \in \R^{2}$.
A solution of $AZ=Y$ has the form $Z=(\tau ,1-\tau c,1-\tau c)^T,\, \tau\in \R$. For such \(Z\) we have 
\[
  \owl(Z)
  = \frac{\norm{Z}_{1}}{\norm{Z}_2}
  = \frac{|\tau|+2|1 - \tau c|}{\sqrt{\tau^2 + 2(1 - \tau c)^2}}.
\]
{For \(c = 0\), the functional assumes a local minimum with value \(\sqrt{2}\) at zero, but the global infimum is given by one for \(\abs{\tau}\to\infty\). However, then \(\spark(A) = 1\) which is not sufficient for recovery of 2-sparse vectors with any method, since there exists a 1-sparse vector in the kernel.
For \(c > 0\) we have \(\spark(A) = 3\) and \(\owl(Z)\) attains minimum at  $\tau = 1/c$, different from the true solution at \(\tau = 0\)} (note, however, that the sparsest solution is at $\tau=1/c$, and the true solution $X$ is not the sparsest here).
Yet, by adding another column to the matrix $X$ to make it 2-row sparse with rank two, 
\[X=\begin{pmatrix}
0 & 0\\
1 & 1 \\
1 & -1
\end{pmatrix}
\quad\text{and the observation }
Y = AX=
\begin{pmatrix}
1 & 1\\
1 & -1
\end{pmatrix},
\]
the problem~\eqref{prob:nonconvex} has a unique solution that is equal to $X$. This follows from the fact that $r = s = 2$ and $s < \spark(A) = 3$. By applying Theorem~\ref{thm:unique_is_truth}, we can conclude that the true solution is recovered.
\end{example}

{
The next example, in a similar spirit, demonstrates that the functional can even fail to have a minimizer for large \(\spark(A)\), if the rank is smaller than the sparsity.
}

\begin{example}
\label{ex:ssp_fail_2} For a $c > 0$, let
\begin{equation*}
A = 
\begin{pmatrix}
c & 1 & 0 &  0 & 0 \\
c & 0 & 1 & 0 & 0\\
0 & 0 & 0 & 1 & 0\\
 0 &0 & 0 &0 &1
\end{pmatrix},
\quad 
X=\begin{pmatrix}
0\\
0\\
0\\
1 \\
1
\end{pmatrix}
\quad 
\text{  and  }\quad 
Y = AX=
\begin{pmatrix}
0\\
0\\
1 \\
1
\end{pmatrix}.
\end{equation*}
Any solution of $AZ=Y$ has the form $Z=(\tau,-\tau c,-\tau c, 1 ,1)^T,\, \tau\in \mathbb{R}$. For such \(Z\), we have 
\[
\owl(Z) 
= \frac{\norm{Z}_{1}}{\norm{Z}_2}
= \frac{2+(1+2c)|\tau|}{\sqrt{2 + (1+2c^2)\tau^2 }}.
\]
The critical points in terms of $\tau$ are $\pm (1+2c)/(1+2c^2)$, $\pm \infty$, and $0$ with local minimum at the latter. However, the local minimum at zero is only a global one for $c > 1/4$,
due to \(\lim_{|\tau|\to \infty} \owl(Z) = (1 + 2c)/\sqrt{1 + 2 c^2} > \sqrt{2} = \owl(Z)\rvert_{\tau=0}\),
and for \(c < 1/4\) the infimum is approached for $\tau \to \pm \infty$.

As before, if we add another linearly independent column with the same sparsity pattern to $X$, say the column 
\(X_2 = ( 0,0,0,1,-1 )^T\), then $\rank(X)=2$ and with Theorem~\ref{thm:unique_is_truth} and $s=2<\spark(A)=3$ the problem~\eqref{prob:nonconvex} has a unique solution equal to $X$.
\end{example}


\section{Existence of a global minimizer}\label{sec:glob_min_noiseless}
Here we  turn to the question of showing that the global minimizers of the noise free~\eqref{prob:nonconvex} and the noisy unconstrained~\eqref{prob:nonconvex_reg} problems indeed exist under certain conditions on $A$. They do not have to exist in general.

To illustrate this fact, let us consider for the moment the case \(K=1\), where the functional \(\owl\) is identical to the \(\ell_1/\ell_2\) functional~\eqref{eq:l1overl2}.
For the existence of the global minimum of the cost function with the $\ell_1/\ell_2$ regularizer, it is often required that $A$ satisfies the $s$-spherical section property. An example can be found in \cite[Theorem 3.4]{doi:10.1137/20M1355380}. In this section we aim to extend these results to joint sprase recovery problem with  the $\owl$ regularizer. The failure of minimizers to exist without appropriate assumptions has been illustrated in the previous section.







\begin{definition}
The rank-$r$ singularity set of the matrix $A$ is defined as
\[\Omega_{r} (A)=\{Z\in\R^{N\times K}:\; \rank(Z)\geq r, \, \rank(AZ)<\rank(Z)\}.\]
\end{definition}
This set contains all matrices of rank \(k\) where some linear combination of the columns is in the kernel of \(A\).
Notice that the rank-$1$ singularity set for $K=1$ is the nontrivial part of the kernel:
For \(K = 1\) we have \(\Omega_1(A) = \ker (A) \setminus \{\,0\,\}\). This justifies the following definition as an extension of the \(s\)-spherical section property (cf.~\cite{vavasis2009derivation, doi:10.1137/20M1355380}).

\begin{definition}[$(r,s)$-spherical section property]
\label{def:SSP}
Let matrix $A\in \R^{M\times N}$. We say that $A$ satisfies the $(r,s)$-spherical section property or $(r,s)$-SSP if, for all $Z \in \Omega_r(A)$,
\begin{align*}
    \owl(Z) > 
    \sqrt{rs}.
\end{align*}
\end{definition}

\begin{lemma}\label{lem:iota}
The following infimum is attainable
\[\iota_r(A)=\inf_{Z\in \Omega_r(A)} \owl(Z). \]
Moreover, for $r<\spark(A)$,  $\iota_r(A)>r$.
\end{lemma}
\begin{proof}
Let $Z\in \Omega_r(A)$ and consider $Z=U\Sigma V^T$ compact singular value decomposition, where $U \in \R^{N\times r^\prime}$,  $V\in \R^{K\times r^\prime}$ with $U^T U = V^TV= I_{r'}$ where \(r'\geq r\) is the rank of \(Z\). Here $\Sigma\in \R^{r^\prime \times r^\prime}$ is the diagonal matrix of nonzero singular values.
Note that $Z (Z^TZ)^{\dagger/2} = UV^T$, therefore 
\[\rank (AU)=\rank (AUV^T) =\rank (AZ (Z^TZ)^{\dagger/2})\leq \rank AZ< \rank{Z}=\rank{UV^T}\] and thus
{\(U \in \Omega_r'(A)\)}. We define  the set
\begin{align*}
    \hat{\Omega}_{r'}(A) &= \Omega_{r'}(A) \cap \left\{U \in \R^{N\times r^\prime}:\; U^T U = I_{r'} \right\}\\
    &=
\left\{U \in \R^{N\times r^\prime}:\; U^T U = I_{r'} \text{ and } \rank{AU} < r' \,\right\}
\end{align*}
where \(I_{r'}\) is the $r^\prime\times r^\prime$ identity matrix. 

By the previous argument, we know that for any \(Z \in \Omega_r\) there exists a \(U \in \hat{\Omega}_{r'}(A)\) with \(r'\geq r\) such that
\[
\owl(Z) = \owl(UV^T)=\owl(U) = \norm{U}_{2,1}.
\]
Thus, we have
\[
\inf_{Z\in \Omega_r(A)} \owl(Z) = \inf_{r'\geq r; U \in \hat{\Omega}_{r'}(A)} \norm{U}_{2,1} = \min_{r'\geq r} \inf_{U \in \hat{\Omega}_{r'}(A)} \norm{U}_{2,1}.
\]
Moreover, the set \(\hat{\Omega}_{r'}(A) \) is closed and bounded, and thus there exists a \(U^* \in \hat{\Omega}_{r'}(A)\subset \Omega_r(A)\) for some \(r'\geq r\) such that 
\[
\inf_{Z\in \Omega_r(A)} \owl(Z) = \owl(U^*) = \norm{U^*}_{2,1}.
\]

Now, let $Z\in \Omega_r(A)$.  From Lemma~\ref{lem:Psibound}, $\owl(Z)\geq \rank(Z)\geq r$. Moreover, from the same lemma, $\owl(Z)=r$ if and only if  $\|Z\|_{2,0}=\rank(Z)=r$.  But  $\rank(AZ)<\rank(Z)$, therefore there exists a non-zero  linear combination of  the columns of $Z$ that is in the kernel of $A$. Such vector will have to be $r$-sparse as a linear combination of $r$-sparse vectors with the same sparsity pattern.  But there cannot be any $r$-sparse vector in the kernel of $A$ otherwise the conditions of Theorem~\ref{thm:rankrecequiv} will be violated (we can construct an $r$-sparse rank $r$ matrix that cannot be recovered from the measurements). Thus,  for any  $Z\in \Omega_r(A)$, $\owl(Z)> r$. Since the infimum is attainable, we have that $\iota_r(A)>r$.
\end{proof}

\begin{example}\label{ex:ssp_gets_better} To illustrate the advantage of larger \(r\) in a specific instance, consider $N=3, K=2$.
Let 
\[A = 
\begin{pmatrix}
c & 1 & 0 \\
c & 0 & 1 
\end{pmatrix} \]
for $c\geq 0$. Notice that, 
\[Z = 
\begin{pmatrix}
1 & 1\\
-c & -c \\
-c & -c
\end{pmatrix} \in \Omega_{1} (A)\]
and 
\[
  \owl(Z)=\frac{1+2c}{\sqrt{1+2c^2}}\leq \sqrt{2}
\]
when $c\leq 1/2$.
Here we used the fact that $\owl(Z)$ is the $\ell_{2,1}$ norm of the matrix with any orthonormal basis in the space of the columns of the matrix $Z$. So, in this case,  it is simply the $\ell_1/\ell_2$ functional on the first column.
Therefore, the matrix $A$ does not fullfill $(1,2)$-SSP when $c\leq 1/2$.
However, when $c>0$,  $2<\spark(A)=3$ thus the $(2,2)$-SSP is satisfied due to Lemma~\ref{lem:iota}.
When $c=0$, then it is not $(2,2)$-SSP which was expected. For
\[Z = 
\begin{pmatrix}
1 & 0\\
0 & 1 \\
0 & 0
\end{pmatrix}\in \Omega_{2} (A), \]
we have $\owl(Z)=2$.
\end{example}

\subsection{The noiseless constrained problem}

\begin{lemma}
\label{lemma:unbounded}
Let \(X \in \mathbb{S}(N,K,r)\) with \(\|X\|_{2,0}=s\) and \(Y = AX\), where $ s<\spark(A)$. Assume there exists an unbounded minimizing sequence $\{Z_k\}$ of~\eqref{prob:nonconvex}. Then 
\begin{equation*}
\iota_r(A) \leq  \min_{Z\in \R^{N\times K}, \; AZ=Y }  \owl(Z). 
\end{equation*}
\end{lemma}
\begin{proof}
Let \(Z_k\) be minimizing sequence. From the proof of Proposition 1 in \cite{petrosyan2019reconstruction}, $ s<\spark(A)$ implies that $A$ preserves the rank of any $s$-row sparse matrix.  Thus $\rank(Z_k)\geq \rank(AZ_k)= \rank(Y)=\rank(X) =r$. 

Without loss of generality (by switching to a subsequence), we may assume that all the $Z_k$ have the same rank.
Assuming the compact SVD of $Z_k$ is
$ Z_k = U_k \Sigma_k V^T_k$, we have 
$
Z_k (Z_k^T Z_k)^{\dagger/2} = U_k  V^T_k.
$
Moreover, along a subsequence of \(k\) the largest singular value of \(\Sigma_k\) diverges and the corresponding left and right singular factors $U_k$ and $ V_k$ converge due to compactness by switching to another subsequence.  Let  $\hat{Z}$ be the limit of the resulting subsequence of $Z_k(Z_k^T Z_k)^{\dagger/2}$. $\hat{Z}$ has the same rank as the matrices in the subsequence (from the fact that $U_k$ and $V_k$ are orthogonal)  and
since \(\rank(Z_k) \geq r\), we deduce that
 \(\rank(\hat{Z}) \geq r\).


 Next, we  show that $\rank (A\hat Z)<\rank(\hat Z)$.
Let $\hat{v}^1$ be the right singular vector and $\hat{u}^1$ be the left singular vector corresponding to the largest singular value $\sigma^1$ of $\hat Z$.
Note that \[\hat{Z}(\hat{v}^1)^T = 
\lim_{k\to\infty} Z_k(v_k^1)^T (\sigma^1_k)^{-1}
\] where $\hat{v}_k^1$ is the right singular vector and $\hat{u}_k^1$ be the left singular vector corresponding to the largest singular value $\sigma_k^1$ of $Z_k$.
On the other hand 
\[
A \hat{Z} (\hat{v}^1)^T = \lim_{k\to \infty}  AZ_k(v^1_k)^T (\sigma^1_k)^{-1} = 0
\]
since \(AZ_k(v^1_k)^T\) is bounded. Note that
\(
A \hat{Z} = A(\hat{Z} - \hat{u}^1(\hat{v}^1)^T)\) and \(
\rank(\hat{Z} - \hat{u}^1(\hat{v}^1)^T) < \rank(\hat{Z})
\).  Therefore $\hat Z\in  \Omega_r(A)$.
 Using Proposition~\ref{cor:phi-lsc},
 \[\lim_{k\to\infty} \owl(Z_k)
\geq \owl(\hat{Z})\geq \iota_r(A).
\qedhere
\]
\end{proof}

\begin{proposition}\label{prop:noiseless_exist}
Let \(X \in \mathbb{S}(N,K,r)\) with \(\|X\|_{2,0}=s\), \(Y = AX\). Assume $s<\spark(A)$ and $A$ satisfies the $(r,s)$-SSP condition or, equivalently, 
\[ r\leq s < \min\{\iota^2_r(A)/r, \spark(A)\}.\]
Then any (globally) minimizing sequence of~\eqref{prob:nonconvex} is bounded and the set of optimal solutions of~\eqref{prob:nonconvex} is nonempty.
\end{proposition}

\begin{proof}
Let 
\[\nu = \inf_{Z\in \R^{N\times K}, \; AZ=Y }  \owl(Z). 
\]
 Note that  $\owl(X)\ge \nu$.
 We have 
$
\| X (X^T X)^{\dagger/2} \|^2_{\fro} = r$ (the singular values of $X (X^T X)^{\dagger/2}$ are all equal to 1),
therefore, 
\begin{align*}
\|X\|_{2,0}^{1/2} & = \| X (X^T X)^{\dagger/2} \|_{2,0}^{1/2} \ge \frac{\| X (X^T X)^{\dagger/2} \|_{2,1}}{\| X (X^T X)^{\dagger/2} \|_{\fro}} \ = \  \frac{\owl(X)}{\sqrt{r}} \ge \frac{\nu}{\sqrt{r}}.
\end{align*}
{For every} $Z \in \Omega_r(A)$, we have 
\begin{align*}
  \owl(Z) \geq \iota_r(A)>   \sqrt{r s }  = \sqrt{r} \|X\|_{2,0}^{1/2} 
  \ge \nu .
\end{align*}
This, combined with Lemma~\ref{lemma:unbounded}, implies that there exists no unbounded minimizing sequence $\{Z_k\}$ of~\eqref{prob:nonconvex_reg}.
Without loss of generality, by passing to a subsequence if necessary, we can assume $Z_k\to \bar Z \in \mathbb{R}^{N\times K}$. From Corollary~\ref{cor:phi-lsc}, $\liminf\limits_{k\to\infty} \owl(Z_k)\geq  \owl(\bar Z)$ therefore, $\bar Z$ is the global minimum of~\eqref{prob:nonconvex}, completing the proof. 
\end{proof}

\begin{corollary}
Let  \(Y = AX\) where $X$ is an $s$-row sparse rank $s$ matrix with $s < \spark(A).$
Then any (globally) minimizing sequence of~\eqref{prob:nonconvex} is bounded and the set of optimal solutions of~\eqref{prob:nonconvex} is nonempty.\end{corollary}
\begin{proof}
We have $r=s$. From Lemma~\ref{lem:iota}, we have that $\iota_s(A)>s$, so the conditions of Theorem~\ref{prop:noiseless_exist} are satisfied. 
\end{proof}

\subsection{The noisy regularized problem}
Next, we turn to the question of showing that the minimizers of the problem~\eqref{prob:nonconvex_reg} indeed exist under certain conditions on $A$.

\begin{lemma}
\label{lemma:unconstrained_unbounded}
Let \(X \in \mathbb{S}(N,K,r)\) with \(\|X\|_{2,0}=s\) and  \(\norm{Y-AX} \leq \delta\), where $ s<\spark(A) $. Assume $\delta$ and $\alpha$ are
 sufficiently small and  assume there exists an unbounded minimizing sequence $\{Z_k\}$ of~\eqref{prob:nonconvex_reg}. Then  
\begin{equation*}
\iota_r(A)\leq  \min_{Z\in \R^{N\times K} } \left( \owl(Z)+\frac{1}{2\alpha} \|AZ-Y\|_\fro^2 \right). 
\end{equation*}
\end{lemma}
\begin{proof}
Let \(Z_k\) be minimizing sequence for the  right side. For \(k\) large enough,
\[
\owl(Z_k) + \frac{1}{2\alpha} \norm{AZ_k - Y}^2_{\fro}
\leq 2  (\owl(X) + \frac{1}{2\alpha} \norm{AX-Y}^2_{\fro})  
\]
thus 
\[
\frac{1}{2}\norm{AZ_k - Y}^2_{\fro}
\leq 2  (\alpha\owl(X) + \frac{1}{2} \norm{Y-AX}^2_{\fro}) - \alpha\owl(Z_k)
\leq 2 \alpha \sqrt{r}\sqrt{s} +
 \delta^2.
\]
This implies that \(\rank(Z_k) \geq r\) for \(\delta\), \(\alpha\) small enough. Indeed,  $s<\spark(A)$ implies that $A$ preserves the rank of any $r$-row-sparse matrix (see Proposition 1 in \cite{petrosyan2019reconstruction}). If $\delta$ is small enough then $\rank(Y)\geq \rank(AX)=r$. And when both $\alpha $ and $\delta$ are small enough, $\rank(AZ_k)\geq \rank(Y)\geq r$. On the other hand $\rank(Z_k)\geq \rank(AZ_k)\geq r$.

Like in the proof of Lemma~\ref{lemma:unbounded},
without loss of generality, all the $Z_k$ have the same rank and  
$
Z_k (Z_k^T Z_k)^{\dagger/2}
$
converge to a  $\hat{Z}\in \Omega_r$. 
Therefore,
 \[\lim_{k\to\infty} \left(\owl(Z_k) + \frac{1}{2\alpha} \|AZ_k-Y\|_\fro^2 \right) 
\ge \liminf_{k\to\infty} \owl(Z_k) 
\geq \owl(\hat{Z})\geq \iota_r(A).
\qedhere
\]
\end{proof}

\begin{proposition}
\label{prop:unconstrained_bounded}
Let \(X \in \mathbb{S}(N,K,r)\) with \(\|X\|_{2,0}=s\) and \(\norm{Y-AX} \leq \delta\). Assume, $ s<\spark(A)$ and $A$ satisfies the $(r,s)$-SSP or, equivalently, 
\[ r\leq s < \min\{\iota^2_r(A)/r, \spark(A)\}.\]
Then, for sufficiently small $\delta^2/\alpha$ and $\alpha$, any (globally) minimizing sequence of~\eqref{prob:nonconvex_reg} is bounded and the set of optimal solutions of~\eqref{prob:nonconvex_reg} is nonempty.

\begin{proof}
 Since $\owl(X) + \dfrac{1}{2\alpha} \|AX-Y\|_\fro^2 \ge \nu$, we have that $\owl(X)\ge \nu-\dfrac{\delta^2}{2\alpha}$.
 We have 
$
\| X (X^T X)^{\dagger/2} \|^2_{\fro} = r$,
therefore, 
\begin{align*}
\|X\|_{2,0}^{1/2} & = \| X (X^T X)^{\dagger/2} \|_{2,0}^{1/2} \ge \frac{\| X (X^T X)^{\dagger/2} \|_{2,1}}{\| X (X^T X)^{\dagger/2} \|_{\fro}} \ = \  \frac{\owl(X)}{\sqrt{r}} \ge \frac{1}{\sqrt{r}}\left(\nu -\frac{\delta^2}{2\alpha}\right).
\end{align*}
Thus, we have
\begin{align*}
  \iota_r(A)>   \sqrt{ \iota_r(A) }\sqrt{r}  = \frac{\sqrt{ \iota_r(A) }}{s}\sqrt{r} \|X\|_{2,0}^{1/2} 
  \ge \frac{\sqrt{ \iota_r(A) }}{s}\left( \nu -\frac{\delta^2}{2\alpha}\right)>\nu
\end{align*}
when  $\delta^2/\alpha$ is sufficiently small. 
This, combined with Lemma~\ref{lemma:unbounded}, implies that there exists no unbounded minimizing sequence $\{Z_k\}$ of~\eqref{prob:nonconvex_reg}.
Without loss of generality, by passing to a subsequence if necessary, we can assume $Z_k\to \bar Z \in \mathbb{R}^{N\times K}$. Then $\|AZ_k-Y\|_\fro^2 \to \|A\bar Z-Y\|_\fro^2 $ and, from Corollary~\ref{cor:phi-lsc}, $\liminf\limits_{k\to\infty} \owl(Z_k)\geq  \owl(\bar Z)$. Therefore, $\bar Z$ is the global minimum of~\eqref{prob:nonconvex_reg}, completing the proof. 
\end{proof}

\begin{corollary}
Let  \(Y = AX\) where $X$ is an $s$-row sparse rank $s$ matrix with $s < \spark(A).$
Then, for sufficiently small $\delta^2/\alpha$ and $\alpha$, any (globally) minimizing sequence of~\eqref{prob:nonconvex_reg} is bounded and the set of optimal solutions of~\eqref{prob:nonconvex_reg} is nonempty.
\end{corollary}

\end{proposition}

\section{The approximate problem}\label{sec:approx}

The functional \(\owl\) is highly discontinuous (recall that this is a necessary condition for rank-awareness), especially in the case of large \(K\), and not directly amenable to computations, unless it is restricted to the submanifold \(\St(N,K,r) \subset \R^{N\times K}\) of matrices of given rank \(r\).
In the noise free case, we can usually deduce the appropriate rank from the data \(Y\), since \(\rank(X) = \rank(Y)\) is a necessary condition for recovery.
Using a decomposition \(Y = Y' Q\) with \(QQ^T = I \in \R^{r\times r}\), we can then solve the problem with data \(Y' = YQ^T \in \R^{N\times r}\) for the solution \(Z' \in \St(N,r) \subset \R^{N\times r}\) and obtain a reconstruction of \(X\) as \(Z'Q\); see \cite{petrosyan2019reconstruction}.

We do not know what the rank of $X$ should be from the data \(Y = AX + W\) in the noisy case, since usually we would expect the noise \(W\) to have full rank. If we preprocess the data $Y$ via its SVD, we can still ensure that the data has full rank. However, in the presence of noise, a number of small singular values may potentially be associated to just noise and an arbitrary cutoff for the small singular values can be difficult to determine in practice ahead of solving the problem.
Additionally, even when we happen to guess the correct rank, an iterative solution algorithm can run into a singularity during one of the iterations. This challenge can be addressed by carefully selecting the step-size; however, this may still yield an ill-conditioned matrix, potentially causing the updates to stagnate. A proper relaxation of the problem may take care of this difficulty.

For this and other reasons to be discussed later, we consider a generalization of the \(\owl\)-functional with a parameter \(\gamma \in [0, 1]\):
\[
  \Psi_\gamma(Z)=\norm{Z\left(\gamma I + (1-\gamma) Z^T Z\right)^{\dagger/2}}_{2,1},
\]
and a generalized version of problem~\eqref{prob:nonconvex_reg}: 
\begin{equation}
\label{prob:nonconvex_reg_gamma}
\min_{Z\in \R^{N\times K}} \left[\Psi_\gamma(Z) + \frac{1}{2\alpha}\norm{A Z - Y}^2_\fro \right].
\end{equation}
The numbers $\gamma\in [0,1]$ and $\alpha>0$ are the hyper-parameters of this problem.
In particular, $\gamma$ controls the convexity of the problem, where $\gamma=1$ reduces it to the convex $\ell_{2,1}$ norm regularized problem and $\gamma=0$ corresponds directly to the $\owl$ regularization.
{For \(\gamma \in (0,1)\), we obtain a new ``relaxed'' optimization problem.} 

As a first verification of the appropriateness of this relaxed problem we will show that the (global) solutions of this problem converges to the (global) solution of the problem in~\eqref{prob:nonconvex_reg} as $\gamma\to 0$ (Proposition~\ref{prop:limit_minimizers}).
Let 
\[
  J_\gamma(Z) = \Psi_\gamma(Z)+\frac{1}{2\alpha} \|AZ-Y\|_\fro^2.
\]
be the associated objective functional. As a corollary of Proposition~\ref{convergence:Psi} on the \(\Gamma\) convergence of \(J_\gamma\) we obtain the following:
\begin{proposition}\label{prop:limit_minimizers}
Any cluster point of minimizers \(\bar{Z}_\gamma\) of \(J_\gamma\) with $0 < \gamma\to 0$ is a  minimizer of~\eqref{prob:nonconvex_reg}. 
\end{proposition}

An additional advantage of the relaxed problem, which will turn out to be useful for numerical optimization, is that its minimizers can always be characterized by first-order conditions \(\bar{Z}\) with the help of the Clarke subdifferential.
It is defined for any locally Lipschitz function; see \cite[Chapter 10]{clarke2013functional}. According to Proposition~\ref{prop:relax_basic_props} we have:
\begin{proposition}
$J_\gamma(Z)$ is locally Lipschitz continuous for any \(\gamma \in (0,1]\).
For \(\gamma = 0\) it is only locally Lipschitz continuous on neighborhoods of points \(Z\) where \(Z^TZ\) is invertible (\(\rank(Z) = K\)).
\end{proposition}

For the following, we introduce the  matrix \(W = (\gamma I + (1-\gamma)Z^TZ)^{\dagger}\)
and note that $\Psi_\gamma(Z) = \|Z\|_{W,1}$.
When either \(\gamma > 0\) or  \(Z^TZ\) is invertible, the matrix \(M = \gamma I + (1-\gamma) Z^T Z\) is symmetric positive definite and thus \(W = M^{-1} = \left(\gamma I + (1-\gamma) Z^T Z\right)^{-1}\).

The version of Fermat's rule for Clarke subdifferential states that \cite[Exercise 10.7]{clarke2013functional}, at a local minimimum  $\bar Z$ of $J_\gamma$,   it holds
\begin{equation}
\label{eq:fermat}
0 \in \partial_C J_\gamma(\bar{Z}) = \partial_C \Psi_\gamma(\bar{Z}) + \frac{1}{\alpha} A^T(A\bar{Z} - Y).
\end{equation}
The additive property does not always hold for the Clarke subdifferential, only an inclusion, but it does in this case due to \cite[Exercise 10.16]{clarke2013functional}.
To characterize the subdifferential, we provide a convex approximation to \(\Psi_\gamma\) at a given reference point.
This is also known as a model function in nonconvex, nonsmooth optimization; cf., e.g., \cite{NollProtRondepierre:2008,Noll:2010}.
\begin{definition}
  \label{def:model_Psi}
  Let \(\widehat{Z}\) be given with corresponding invertible \(\widehat{M} = \gamma I + (1-\gamma) \widehat{Z}^T \widehat{Z}\) and \(\widehat{W} = \widehat{M}^{-1}\).
  The convex model function is defined as
  \begin{equation}
    \label{eq:model_Psi}
    m^{\Psi}_{\widehat{Z}}(Z) :=
    \norm{Z}_{\widehat{W},1}
    + \langle \widehat{Z}\widehat{\Lambda}, Z - \widehat{Z}\rangle
    \quad\text{where }
    \widehat{\Lambda} 
    = -\sum_{n=1}^N\frac{1-\gamma}{\norm{\widehat{z}_n}_{\widehat{W}}} \widehat{W} \widehat{z}_n \widehat{z}_n^T \widehat{W}.
  \end{equation}
\end{definition}
With a characterization of the Clarke subdifferential of $\Psi_\gamma(Z)$
the above optimality condition can be equivalently stated as follows:
\begin{theorem}
\label{thm:opt_cond_model}
Let \(\bar{Z}\) be a local minimum of \(J_\gamma\), where either \(\gamma > 0\) or \(\Bar{Z}^T\Bar{Z}\) invertible.
Then, in terms of the Clarke subdifferential, the following first order necessary conditions hold:
\begin{align*}
     \bar{z}_n\left(\frac{1}{\norm{\bar{z}_n}_{\bar{W}}} \bar{W} - \sum_{i=1}^N \frac{1-\gamma}{\norm{\bar{z}_i}_{\bar{W}}} \bar{W} \bar{z}_i \bar{z}^\top_i \bar{W} \right) + \frac{1}{\alpha}(A^T(A\bar{Z}-Y))_n &= 0 &&\text{for } \bar{z}_n \neq 0,\\
    \norm{(A^T(A\bar{Z}-Y))_n\bar{W}^{-1}}_{\bar{W}} = \norm{(A^T(A\bar{Z}-Y))_n}_{\bar{W}^{-1}} &\leq \alpha
        &&\text{for } \bar{z}_n= 0,
\end{align*}
where $\bar z_n$ is the $n$-th row of $\bar Z$ and \(\bar W = (\gamma I+(1-\gamma)\bar Z^T\bar Z)^{-1}\).
\end{theorem}
\begin{proof}
From Theorem~\ref{thm:clark_subdifferential}, the Clarke subdifferential is given as the convex subdifferential
\[
  \partial_C \Psi_\gamma(\bar Z)
  = \partial m^{\Psi}_{\bar Z}(\bar Z)
  =\partial \norm{\bar Z}_{\bar{W},1} + \bar{Z}\bar{\Lambda},
\]
of the model function at \(\bar{Z}\).
The result follows now from~\eqref{eq:fermat} and the well-known explicit characterization of the convex sub-differential of the \(\ell_{\bar{W},1}\)-norm $\partial \norm{\cdot}_{\bar{W},1}$ for given weight \(\bar{W}\).
\end{proof}

\section{A proximal gradient based algorithm}\label{sec:algorithm}

We assume first that \(\alpha>0\) and \(\gamma\geq 0\) are given, and solve the problem~\eqref{prob:nonconvex_reg_gamma}.
We first derive the idea of the algorithm using a proximal Lagrange method.
Then, we will provide a convergence analysis using the model function~\eqref{eq:model_Psi}, and then we will briefly discuss the adaptive selection of the parameters \(\gamma\) and \(\alpha\).

We note that in the case \(\gamma = 0\), the algorithm may run into cases where \(M = Z^TZ\) is not invertible, which corresponds to a discontinuity in \(\owl\).
We will disregard that possibility for the moment and discuss mitigation strategies in the context of stepsize selection. In general, if the sought after solution of the problem does not fulfill \(\bar{M} = \bar{Z}^T\bar{Z}\) invertible, we will tacitly assume that \(\gamma> 0\) and rely on Proposition~\ref{prop:limit_minimizers}.

\subsection{Proximal Lagrange derivation}
To derive an algorithm, we first aim to isolate the nonsmooth and nonconvex parts of the functional and rewrite the problem~\eqref{prob:nonconvex_reg_gamma} in the constrained form
\begin{equation}
\label{eq:gamma_problem_constr}
\min_{Z\in \R^{N,K} , W \in \R^{K\times K}} \|Z\|_{W,1}+\frac{1}{2\alpha} \|AZ-Y\|^2_\fro\quad
\text{subject to } W = (\gamma I+(1-\gamma)Z^TZ)^{-1},
\end{equation}
where we have introduced a separate variable for the weight \(W\). Note that, similar to the model function~\eqref{eq:model_Psi} postulated above, the functional now contains a weighted convex group sparse norm, if \(W\) is fixed.
To solve the problem above, we consider the Lagrange function
\begin{align*}
\mathcal{L}(Z,W,\Lambda)
&= \sum_{i=1}^N \|z_i\|_W + \frac{1}{2\alpha} \|AZ-Y\|_\fro^2 + \frac{1}{2(1-\gamma)} \tr(\Lambda^T(\gamma I + (1-\gamma)Z^TZ - W^{-1}))\\
&=
\sum_{i=1}^N\|z_i\|_W + \mathcal{M}(Z,W,\Lambda),
\end{align*}
which is split into a part that is convex in \(Z\) and smooth in \(W\) and
a second part \(\mathcal{M}\) that collects the remaining differentiable parts of the Lagrange function \(\mathcal{L}\).
The factor \(1/(2(1-\gamma))\) is added to provide a convenient scaling of the multiplier in the subsequent analysis.

Additionally, we consider the Hilbert space on \(\R^K\) induced by the \(W\) inner product of the form $\langle z,z' \rangle_W= z^TWz'$,  $z,z' \in \R^{K}$.
We note that it also induces an inner product $\langle Z',Z \rangle_{W}= \tr (Z'WZ^T)$ for matrices $Z,Z'\in \R^{N\times K}$.
The idea essential to our strategy is to compute gradients with respect to $Z$ in this inner product.
Given a function $F:\R^{N\times K}\to \R$, $\nabla^WF(Z)=\nabla F(Z)W^{-1}$ denotes the gradient of $F$ in the $W$-Hilbert space, i.e.\ it is the matrix for which the directional derivative can be written as
\[
F^\prime (Z;\Delta Z)
= \lim_{h\to 0}\frac{F(Z+h\Delta Z)-F(Z)}{h}
= \langle \nabla^WF(Z),\Delta Z \rangle_W.
\]
Also, we define the $W$-proximal operator as 
\begin{equation}\label{eq:W-prox}
    \prox_{W,\sigma}(Z')=\arg\min_Z\|Z\|_{W,1}+\frac{1}{2\sigma}\|Z-Z'\|^2_{W,2}.
\end{equation}
Since the choice of the inner product agrees with the vector norm used in the definition of \(\norm{\cdot}_{W,1}\), it can be checked that the closed form solution is given by
\[
\left[\prox_{W,\sigma}(Z)\right]_n = z_n \max\left\{0,\; 1-\frac{\sigma}{\|z_n\|_{W}}\right\}
= z_n \left(1 - \frac{\sigma}{\max\{\sigma,\|z_n\|_{W}\}}\right).
\]
Note that this analytic representation would not be possible for the standard inner product.

The main steps of the algorithm, given a previous iterate \(Z_k\), \(W_k\) (\(k\geq 0\)) are now derived as follows:
First, we compute the update for the multiplier \(\Lambda_k\)
as a solution to the Lagrange stationarity condition $\nabla_W  \mathcal{L}(Z_{k},W_{k},\Lambda_{k}) = 0$.
We observe that 
\[
  \nabla_W \mathcal{L}(Z_k,W_k,\Lambda) 
  = \frac{1}{2}\sum_{\{\, n\colon z^k_n \neq 0 \,\}} \frac{z^k_{n} (z^k_{n})^T}{\sqrt{(z^k_n)^T W_k z^k_n}}
   + \frac{1}{2 (1-\gamma)}W_k^{-1}\Lambda W_k^{-1},
\]
and the equation can be uniquely solved for \(\Lambda = \Lambda_k\) as given is step 4 of Algorithm~\ref{alg:practical} below; cf.\ Definition~\ref{def:model_Psi}. 

Second, with given \(Z_k\), \(W_k\) and \(\Lambda_k\), we select a stepsize \(\sigma_k > 0\) and update
$Z_{k+1} = \prox_{W_{k},\sigma_k}(Z_{k} - \sigma_k G_k)$, where \(G_k = \nabla_Z \mathcal{M}(Z_{k},W_{k},\Lambda_{k})W_{k}^{-1}\) is the \(W_k\)-gradient of the smooth part of the Lagrange function given as
\[
  \nabla_Z \mathcal{M}(Z,W,\Lambda) = \frac{1}{\alpha} A^T(AZ - Y) + Z \Lambda.
\]
In Subsection~\ref{subsec:armijo_stepsize} we introduce and discuss an Armijo line search procedure for finding the  step size.

Finally, we need to update $W$, where we simply set $W_{k+1} = (\gamma I+(1-\gamma)Z_{k+1}^TZ_{k+1})^{-1}$, which solves exactly the constraint. Note that due to this choice, $W_k$ can be eliminated as an explicit iteration variable and only \(Z_0\) needs to be provided to initialize the algorithm.

Our proposed iterative scheme to solving the problem~\eqref{eq:gamma_problem_constr} is summarized with the following steps:  
\begin{algorithm}[H]
  \caption{}
  \label{alg:practical}
 \begin{algorithmic}[1]
   \State Initialize $Z_0$
   \While{not converged}
   \State     $W_{k} = (\gamma I+(1-\gamma)Z_{k}^TZ_{k})^{-1}$
   \State  $\Lambda_{k} = -\sum_{{\{\, n\colon z^k_n \neq 0 \,\}}} \frac{1-\gamma}{\sqrt{(z^k_n)^T W_{k} z^k_n}}W_{k}z^k_n(z^k_n)^T W_{k}$
    \State 
    $G_{k}=(Z_{k}\Lambda_{k}+(1/\alpha)A^T(AZ_{k}-Y) )W_{k}^{-1}$
\State 
    Choose the step size $\sigma_{k} $ with a Armijo line search and set 
    $Z_{k+1} = \prox_{W_{k},\sigma_{k}}(Z_{k} - \sigma_{k}G_{k})$
    \State $k=k+1$.
   \EndWhile \end{algorithmic} 
\end{algorithm}
We note that for \(\gamma=1\),  we have \(W_k = I\), \(\Lambda_k = 0\), and we obtain the classical proximal gradient method for group sparse least squares regression.
Note also that this extension's main additional effort is the computation of the inverse \(W_k\), which is cheap for moderately sized \(K\) compared to \(N\).
  \begin{remark}
    \label{rem:large_K}
  For the case of large \(K\), the effort of direct computation of the inverse \(W_k\) can become prohibitive. {If $Y$ is approximately low rank, this issue can be alleviated by decomposing $Y = Y'Q$, where $Y' \in \mathbb{R}^{M\times r}$ with $r \ll K$ and $Q \in \R^{r\times K}$ with $QQ^T = I \in \R^{r\times r}$, then the optimization problem reduces to a cheaper one where the data matrix $Y$ is replaced by $Y'$ (see Section~\ref{sec:feature_selection} for examples).}
  Generally, an efficient way to compute \(W\) can be given for \(s = \ellzero(Z) < K\) and \(r = \rank(Z) \leq s < K\): First, we decompose  \(Z = P_{I(Z)}Z' = P R Q\), where \(P_{I(Z)} \in \{0,1\}^{N\times s}\) is the selection of columns of the identity matrix corresponding to the support \(I(Z)\) of \(Z\), and the factors \(R \in \R^{s\times r}\) and \(Q\in\R^{r\times K}\) with \(Q\) orthogonal can be computed, e.g., with a QR-decomposition \(Q^T R^T = Z^T P_{I(Z)} = (Z')^T\).
  Then, with \(M = \gamma I + (1-\gamma) Q^TR^TR Q\) we have \(W = M^\dagger = \gamma^{-1}(I - Q^TQ) + Q^T (\gamma I + (1-\gamma) R^TR)^{-1} Q\).
  If we now modify the algorithm to not explicitly store \(W\), but only \(P\), \(Q\) and \(R\), the complexity can be reduced if \(s \ll K\).
  Note that the matrix \(\Lambda_k\) is also in low-rank format as the sum of \(s\) rank-one products.
  In this context, it is critical to ensure that \(s_k = \rank(Z_k)\), remains small from initialization throughout the iterations.
  We postpone a detailed investigation of this to future work.
\end{remark}

\subsection{ Step size selection for descent}\label{subsec:armijo_stepsize}
In this section, we demonstrate that the iteration steps in the proposed algorithm result in the decrease in the loss function and therefore, the algorithm is convergent if the step size \(\sigma_k\) is chosen appropriately.

Let $\widehat Z\in \R^{N\times K}$ be given, which can be thought of as a given (old) iterate \(Z_k\).
Assume $\gamma\in(0,1]$ or $\gamma=0$ and~$\widehat Z^T\widehat Z$ is invertible.
Based on the model function \(m^\Psi_{\widehat{Z}}\) from Definition~\ref{def:model_Psi},
the following function will serve as a model function for the objective function \(J=J_\gamma\):
\begin{equation*}
  \begin{aligned}
    m_{\widehat{Z},L}(Z)
    :&=
    m^\Psi_{\widehat{Z}}(Z)
    + \widehat{F}
     + \frac{1}{\alpha}\langle  A^T(A\widehat{Z} - Y) , Z - \widehat{Z}\rangle
     + \frac{L}{2}\norm{Z-\widehat{Z}}_{\widehat{W},2}^2
    \\
    &=
    \widehat{F}
    + \norm{Z}_{\widehat{W},1}
    + \langle \widehat G, Z - \widehat{Z}\rangle_{\widehat{W}} 
    + \frac{L}{2}\norm{Z-\widehat{Z}}_{\widehat{W},2}^2
  \end{aligned}
\end{equation*}
where 
\begin{equation*}
  \widehat{F} = \frac{1}{2\alpha}\norm{A\widehat{Z} - Y}_\fro^2,
  \quad
  \widehat G = (\widehat{Z}\widehat{\Lambda} + (1/\alpha)A^T(A \widehat{Z} - Y))\widehat{W}^{-1}.
\end{equation*}
We will use this model as an upper bound of \(J\), where the constant \(L\) has to account for the neglected terms in \(\Psi_\gamma\) as well as in the quadratic fitting term.
The new iterate \(Z_{k+1}\) can then be interpreted as the minimizer of the model function for \(\widehat{Z} = Z_k\) and \(L = 1/\sigma_k\).
Due to the complictated structure of the nonlinearity in \(\Psi_\gamma\) we will not give \(L\) analytically, but identify it computationally.
\begin{proposition}\label{prop:model_func_approx}
Let $\gamma\in(0,1]$ or $\widehat Z^T\widehat Z$ be invertible.
Then $m_{\widehat{Z},L}(Z)$ has the following properties
\begin{enumerate}
    \item \label{part:1_prop:model_func_approx} \(J(\widehat{Z}) = m_{\widehat{Z},L}(\widehat{Z})\),
    \item \label{part:2_prop:model_func_approx}
      \(J(Z) \leq m_{\widehat{Z},\widehat{L}}(Z)\) 
       for all \(\norm{Z - \widehat{Z}}_{\widehat{W},2} \leq \bar{\Delta}\) with some universal \(\bar{\Delta} > 0\)
        and \(\widehat{L}\) only dependent on \(\norm{\widehat{Z}}_\op\),
    \item \label{part:3_prop:model_func_approx}  the convex function \(m_{\widehat{Z},L}\) is minimized (uniquely) at
    \[
    Z^+ = \prox_{\widehat{W},1/L}\left(\widehat{Z} - (1/L) \widehat{G}\right).
    \]
\end{enumerate}
\end{proposition}
\begin{proof}
Part~\ref{part:1_prop:model_func_approx} is straightforward. 
Part~\ref{part:2_prop:model_func_approx} follows with
\begin{align*}
J(Z)
=\;&\Psi_\gamma(Z) + \frac{1}{2\alpha}\norm{A Z - Y}_{\fro}^2\\
=\; &\norm{Z}_{\widehat{W},1}
+ \langle \widehat{Z}\widehat{\Lambda} + (1/\alpha)A^T(A \widehat{Z} - Y), Z - \widehat{Z}\rangle
+ \widehat{F}
+ \frac{1}{2\alpha}\norm{A (Z -\widehat{Z})}_\fro^2 + R(Z,\widehat Z)\\
=\; &
 m^\Psi_{\widehat{Z}}(Z)  +  R_2(Z,\widehat Z)
\end{align*}
and the estimate 
\[
\abs{R_2(Z,\widehat{Z})} 
\leq \frac{\bar{L} + \widehat{L}_A}{2} \|Z-\widehat{Z}\|^2_{\widehat W, 2} = \frac{\widehat{L}}{2} \|Z-\widehat{Z}\|^2_{\widehat W, 2},
\]
with Theorem~\ref{thm:first_order_approx} 
for some $\bar{L} > 0$ assuming \(\norm{Z-\widehat{Z}}_{\widehat{W},2} \leq \bar{\Delta}\),
and the appropriate Lipschitz constant of the quadratic part is given as
\[
\widehat{L}_A 
 = \frac{1}{\alpha} \left(\sup_{\norm{Z}_{\widehat{W},2} \leq 1} \norm{A Z}_{\fro}\right)^2
 = \frac{\norm{A}_\op^2\norm{\widehat{W}^{-1}}_\op}{\alpha}
 = \frac{\norm{A}_\op^2(\gamma + (1-\gamma)\norm{\widehat{Z}}^2_\op)}{\alpha}.
\]

To prove part~\ref{part:3_prop:model_func_approx}, notice that in the definition given by formula~\eqref{eq:W-prox}, if we replace $Z'$ by  $Z^+ = \widehat{Z} - (1/L) \widehat{G}$ and disregard some terms not dependent on $Z$, we arrive at $m_{\widehat{Z},L}(Z)$.
\end{proof}

\begin{corollary}\label{cor:prox_equiv_cond}
       If \(\bar{Z}\) is a local minimum of \(J\) then, for any \(L>0\), \(\bar{Z}\) is the unique minimizer of \(m_{\bar{Z},L}\), and 
       \[
    \bar Z = \prox_{\bar{W},1/L}\left(\bar{Z} - (1/L) \bar{G}\right),
    \]
    where \(\bar{G} = (\bar{Z}\bar{\Lambda}+(1/\alpha)A^T(A\bar{Z}-Y))\bar{W}^{-1}\).
\end{corollary}
\begin{proof}
    The first claim follows from parts~\ref{part:1_prop:model_func_approx} and~\ref{part:2_prop:model_func_approx} in Proposition~\ref{prop:model_func_approx}. Then, the optimality conditions arise from part~\ref{part:3_prop:model_func_approx}  in Proposition~\ref{prop:model_func_approx}, by realizing that \(Z^+ = \bar{Z}\).
\end{proof}

\begin{remark}
  \label{rem:prox_equiv_cond}
    The condition in Corollary~\ref{cor:prox_equiv_cond} is equivalent to the first order conditions in Theorem~\ref{thm:opt_cond_model}.
\end{remark}

\begin{remark}
\label{rem:algorithm_for_K1}
Let us briefly discuss the case \(K=1\) and \(\gamma = 0\), which reduces to the \(\ell_1/\ell_2\) norm regularization.
Here, \(\widehat{W} = \norm{\widehat{Z}}_2^{-2}\),
\(\widehat{\Lambda} = -\norm{\widehat{Z}}_2^{-3} \norm{\widehat{Z}}_1\) and the model function up to the fixed additive constant \(m^\Psi_{\widehat{Z},0}(0)\) and multiplicative constant \(1/\norm{\widehat{Z}}_2\) is given as
\[
\psi_{\widehat{Z},L}(Z)
= \frac{1}{\norm{\widehat{Z}}_2} \left(m^\Psi_{\widehat{Z},L}(Z) - m^\Psi_{\widehat{Z},0}(0)\right)
= \norm{Z}_1
- \Big\langle\frac{\norm{\widehat{Z}}_1}{\norm{\widehat{Z}}_2}\,\frac{\widehat{Z}}{\norm{\widehat{Z}}_2},\; Z \Big\rangle
+ \langle \widehat{H}, Z \rangle
+ \frac{L}{2} \norm{Z - \widehat{Z}}^2_2
\]
with \(\widehat{H} = (\norm{\widehat{Z}}_2/\alpha) A^T(A\widehat{Z} - Y)\). Replacing \(\langle\widehat{H},Z\rangle\) with the indicator function of the constraint \(A Z = Y\), as is appropriate for solving formulation~\eqref{prob:nonconvex} compared to the regularized version~\eqref{prob:nonconvex_reg}, and setting \(\widehat{Z} = Z_k\) and \(Z^{k+1} = \argmin_Z \psi_{Z_k,L}(Z) = \argmin_Z m^\Psi_{Z_k,L}(Z)\) we obtain exactly \cite[Algorithm~2]{9057443} (see also~\cite{doi:10.1137/20M1355380}).
\end{remark}

We now return to the iterative procedure outlined in Algorithm~\ref{alg:practical}, which generates the sequence \(Z_k\).
We will show that the method is a descent method, i.e., \(J(Z_{k+1})\leq J(Z_k)\).
Note that this implies that
\begin{equation}
\label{eq:Zk_bounded}
\norm{Z_k}_{W_k,1} + \frac{1}{2\alpha} \norm{A Z_k - Y}_\fro^2 \leq J(Z_0).
\end{equation}
For the following analysis, we need to ensure that \(Z_k\) remains bounded, which unfortunately does not directly follow from this inequality due to the fact that \(A\) may have nontrivial kernel and \(\Psi_\gamma \leq \sqrt{(1-\gamma)NK}\) (see Proposition~\ref{prop:relax_basic_props}).
Similar to Proposition~\ref{prop:unconstrained_bounded}, a uniform bound on \(Z_k\) could be deduced under appropriate assumptions on \(\alpha\), \(\delta\), \(Z_0\) and \(A\) using the \((r,s)\)-SSP. 
For simplicity, we from now on simply assume that the sequence \(Z_k\) remains bounded:
\begin{equation}
  \label{eq:assu_Z_bounded}
  \norm{Z_k}_\op \leq C_0 \quad\text{for all } k \geq 1.
\end{equation}
This will allow us to bound certain constants uniformly, in particular \(\norm{W_k^{-1}}_\op = (\gamma + (1-\gamma)\norm{Z_k}^2_\op)\).

In order to show that descent can be fulfilled in every iteration as in~\ref{prop:model_func_approx} with \(\widehat{Z} = Z_k\), we have to replace the unknown constant \(L\) in $m_{Z_{k}, L}(Z)$ with a step size $\sigma_{k} \sim 1/\widehat{L}$.
In practice, since this constant depends on the linearization point \(\widehat{Z} = Z_{k}\) and the size of the update \(Z_{k+1} - Z_k\), we select it by an adaptive Armijo-line search procedure described below.
Given a candidate step size $\sigma>0$, we define the candidate next iteration as  
\begin{equation}\label{eq:next_step_linesearch}
    Z^+_\sigma = \prox_{W_{k},\sigma}\left(Z_{k} - \sigma G_{k}\right)
\end{equation}
where $G_{k}$ is defined as before as \(G_{k}=(Z_{k}\Lambda_{k}+(1/\alpha)A^T(AZ_{k}-Y) )W_{k}^{-1}\).
We define the ``predicted decrease'' associated to a candidate stepsize \(\sigma > 0\) as
\begin{align*}
\operatorname{pred}_\sigma(Z_{k})
= \norm{Z^+_\sigma}_{W_{k},1}-\norm{Z_{k}}_{W_{k},1}
           + \langle G_{k}, Z^{+}_\sigma - Z_{k}\rangle.
\end{align*}
 The predicted decrease serves as an optimistic estimate for how much descent in the functional can be expected in the current iteration, and can be used as an indicator for stationarity of \(Z_{k}\) as characterized by the next theorem. 
\begin{proposition}\label{prop:pred_bound}
For any \(\sigma > 0\) it holds 
\[
\operatorname{pred}_\sigma(Z_{k})
= m_{Z_{k},1/\sigma}(Z^+_\sigma) - m_{Z_{k},1/\sigma}(Z_{k})
- \frac{1}{2\sigma}\norm{Z^+_\sigma - Z_{k}}^2_{W_{k},2}
\leq 0
\]
and, thus,
\begin{equation}
\label{eg:prox_pred_ineg}
  \frac{1}{2\sigma}\norm{Z^+_\sigma - Z_{k}}^2_{W_{k},2} \leq -\operatorname{pred}_\sigma(Z_{k}).
\end{equation}
\end{proposition}
\begin{proof}
  The alternate expression for $\operatorname{pred}_\sigma(Z_{k})$  is easy to verify.
  The inequality $\operatorname{pred}_\sigma(Z_{k})\leq 0$  holds since $ m_{Z_{k},1/\sigma}(Z^+_\sigma) \leq m_{Z_{k},1/\sigma}(Z_{k})$ due to part~\ref{part:3_prop:model_func_approx} in Proposition~\ref{prop:model_func_approx}.
\end{proof}
To determine an appropriate step size, we provide a (large) inital guess \(\sigma_{-1} = \sigma_{\max}\),  in each iteration take the largest \(\sigma \in \{\, \beta^j \min\{\,\sigma_{\max},\; \sigma_{k-1}/\beta\,\} \;|\; j = 0,1,2, \ldots \,\}\) for some line-search parameter \(\beta \in (0,1)\) that achieves a fraction of the predicted decrease, i.e.\ such that it holds
\begin{equation}
  \label{eq:linesearch}
   M^+_\sigma = \gamma I + (1-\gamma)(Z^+_\sigma)^TZ^+_\sigma \text{ is invertible and }
J(Z^+_\sigma) - J(Z_k)
\leq \kappa \operatorname{pred}_\sigma(Z_{k}),
\end{equation}
where \(\kappa \in (0,1)\) is an algorithmic constant.
By the next result, for some $\sigma_k$ chosen by the Armijo line search the condition in~\eqref{eq:linesearch} will be satisfied.
\begin{proposition}
\label{prop:ls_Lipschitz}
Given the bounds~\eqref{eq:Zk_bounded} and~\eqref{eq:assu_Z_bounded}, there exists a \(\bar{\sigma}\)
such that~\eqref{eq:linesearch} is fulfilled for any \(\sigma \leq \bar{\sigma}\).
\end{proposition}
\begin{proof}
In order to apply part~\ref{part:2_prop:model_func_approx} of Proposition~\ref{prop:model_func_approx}, we first need a bound on the difference \(Z^+_\sigma - Z_k\).
By a standard estimate for the Prox-operator and the triangle inequality,
\[
  \norm{Z^+_\sigma - Z_k}_{W_k,2} \leq \sigma \norm{G_k}_{W_k,2}
  \leq \sigma\left(\norm{Z_k\Lambda_k}_{W_k^{-1},2} + \frac{1}{\alpha}\norm{A^T(AZ_{k}-Y)}_{W_k^{-1},2}\right).
\]
Furthermore as in Remark~\ref{rem:Lambda},
\[
  \norm{Z_k\Lambda_k}_{W_k^{-1},2} 
  \leq  \norm{Z_kW_k^{1/2}}_{\op}\norm{W_k^{-1/2}\Lambda_kW_k^{-1/2}}_{\fro}
  \leq (1-\gamma)^{1/2} \norm{Z_k}_{W_k,1}
\]
and
\[
\frac{1}{\alpha}\norm{A^T(AZ_{k}-Y)}_{W_k^{-1},2} \leq \sqrt{L_A}\norm{AZ_{k}-Y}_{\fro}.
\]
By~\eqref{eq:Zk_bounded} and~\eqref{eq:assu_Z_bounded}, both of these quantities are bounded
by a universal constant, and thus there is a \(\bar{g}\)
such that \(\norm{Z^+_\sigma - Z_k}_{W_k,2} \leq \sigma \bar{g}\).
Thus, for \(\sigma\) small enough, by Proposition~\ref{prop:model_func_approx} we know that
there exists a constant \(L_0\)  (dependent only on \( \norm{Z_k}_\op \leq C_0\)), such that
\begin{align*}
J(Z^+_\sigma) - J(Z_{k})
&\leq m_{Z_{k},{L}_0}(Z^+_\sigma) - J(Z_{k}) \\
&= m_{Z_{k},{L}_0}(Z^+_\sigma) - m_{Z_{k},{L}_0}(Z_{k}) \\
&= \operatorname{pred}_\sigma(Z_{k})
 + \frac{{L}_0}{2} \norm{Z^+_\sigma - Z_k}^2_{W_k,2}\\
&= \kappa\operatorname{pred}_\sigma(Z_{k})
 + (1-\kappa) \left[\operatorname{pred}_\sigma(Z_{k}) + \frac{1}{2\sigma} \norm{Z^+_\sigma - Z_{k}}^2_{W_k,2} \right]  +\frac{\sigma{L}_0 - (1-\kappa)}{2\sigma} \norm{Z^+_\sigma - Z_{k}}_{W_k,2}.
\end{align*}
The second term is negative due to Proposition~\ref{prop:pred_bound}. The third term is negative for
\(
\sigma \leq (1-\kappa)/{L}_{0}.
\)
\end{proof}

\begin{proposition}\label{prop:iteration_algorithm}
  Consider \(Z_k\) started from a given \(Z_0\) (with either \(\gamma > 0\) or \(Z_0^TZ_0\) invertible)
  by \(Z_{k+1} = Z^+_{\sigma_k}\) from~\eqref{eq:next_step_linesearch} and \(\sigma_k\) chosen according to~\eqref{eq:linesearch}.
  Finally, assume that~\eqref{eq:assu_Z_bounded} is fulfilled.
  Then the following holds:
\begin{enumerate}
    \item \label{part:1_iteration_algorithm}
        The descent in every iteration can be estimated by
        \begin{align*}
           J(Z_{k+1}) - J(Z_{k})
           \leq \kappa \operatorname{pred}_{\sigma_k}(Z_{k}) \leq 0.
        \end{align*}
    \item \label{part:2_iteration_algorithm}  \(\inf_k \sigma_k > 0\).
    \item \label{part:3_iteration_algorithm} For any accumulation point \(Z'\) of sequence $Z_k$
       (with either \(\gamma > 0\) or \((Z')^TZ'\) invertible)
      the first order necessary condition is fulfilled.
\end{enumerate}
\end{proposition}
\begin{proof}
We apply Proposition~\ref{prop:ls_Lipschitz} for part~\ref{part:1_iteration_algorithm}. 
From this, we derive that \(\lim_{k\to \infty} J(Z_k) = \bar{J}\) for some \(\bar{J} \leq J(Z_0)\), since this is a decreasing bounded sequence and also~\eqref{eq:Zk_bounded} is fulfilled.
For part~\ref{part:2_iteration_algorithm}, we note that by construction of the Armijo stepsize and Proposition~\ref{prop:ls_Lipschitz}, we know that \(\sigma_{\max} \geq \sigma_k \geq \min \{\sigma_{\max}, \beta \bar\sigma\}\).
Moreover,
\begin{align*}
\infty > J(Z_0) - \bar{J}
&= \sum_{k=1}^\infty J(Z_{k}) - J(Z_{k}) \\
&\geq - \kappa \sum_{k=1}^\infty \operatorname{pred}_{\sigma_k}(Z_{k}) 
\geq \frac{\kappa}{2}\left(\inf_k\sigma_k\right) \sum_{k=1}^\infty \norm*{\frac{1}{\sigma_k}(Z^+_{\sigma_k} - Z_{k})}^2_{W_{k},2}.  
\end{align*}
Thus, we obtain
\[
\frac{1}{\sigma_k}(Z^+_{\sigma_k} - Z^{k})W_k^{1/2} \to 0.
\]
Let now \(Z'\) be any accumulation point of \(Z_{k}\) and \(\sigma'>0\) an accumulation point of \(\sigma_k\) along the same subsequence and \(W' = (M')^{-1}\) for invertible \(M' = \gamma I + (1-\gamma) (Z')^TZ'\).
 Along this subsequence we have \(Z^+_{\sigma_k} - Z^{k} \to 0\) due to uniform invertibility of all \(W_k\) and \(\sigma_k\) and thus
\[
  Z^+_{\sigma_k} - Z_{k}
  = \prox_{W_{k},\sigma_k}\left(Z_{k} - \sigma_k G_{k}\right) - Z_k
  \to \prox_{W',\sigma'}\left(Z_{k} - \sigma_k G'\right) - Z' = 0,
\]
where \(G' = (Z'\Lambda'+(1/\alpha)A^T(AZ'-Y) )M'\)
by continuity of all expressions (using Proposition~\ref{prop:W_perturb} and Proposition~\ref{prop:rank_one_fro}).
Thus, with Remark~\ref{rem:prox_equiv_cond} the stationarity conditions are fulfilled.
\end{proof}

\begin{remark}
  According to Proposition~\ref{prop:ls_Lipschitz}, the predicted decrease provides both an upper and a lower bound on the decrease after each successful iteration.
  Scaling it by the stepsize to avoid bias (due to small steps), an appropriate stopping criterion is given by
  \begin{equation}
    \label{eq:stopping_prox}
    - \frac{1}{\sigma_k}\operatorname{pred}_{\sigma_k}(Z_{k}) \leq r_{\mathrm{TOL}} J_k, 
  \end{equation}
  where \(r_{\mathrm{TOL}} < 1\) is a relative tolerance, and \(J_k\)
  some reference value of the functional, e.g., \(J_k = J(Z_k)\).
  Moreover, due to~\eqref{eg:prox_pred_ineg} this provides an estimate on the weighted Prox-residual \(r^k_{\prox} = \sigma_k^{-1}(Z^+_{\sigma_k} - Z^{k})W_k^{1/2}\), which characterizes stationarity.
\end{remark}

\subsection{Adaptive choice of the regularization parameter}
\label{sec:adapt_alpha}
In practice, the value \(\alpha\) has to be chosen appropriately, to balance a good fit of the data with the desire for robustness to noise and sparsity.
This is studied extensively in regularization theory for inverse problems; see, e.g.~\cite{Engl2015}.
In this manuscript, we rely on the (Morozov) discrepancy principle, which presumes that a noise level \(\delta > 0\) is known, and \(\alpha\) is adjusted until an (approximate) solution \(Z_\alpha\) of~\eqref{eq:gamma_problem_constr} fulfills
\begin{equation}
  \label{eq:discrepancy}
  \tau_1\, \delta
  \leq \norm{AZ_\alpha-Y}_\fro
  \leq \tau_2\, \delta
\end{equation}
for some \(0< \tau_1 \leq 1 < \tau_2\).
To achieve this, we use the following procedure:
\begin{enumerate}
\item Solve~\eqref{eq:gamma_problem_constr} by several steps of Algorithm~\ref{alg:practical}
  for given fixed $\alpha$ until~\eqref{eq:stopping_prox} is fulfilled.
\item If~\eqref{eq:discrepancy} is violated, update $\alpha$
  (decrease if the fit is too large, increase otherwise) and go back to step~1.
\end{enumerate}
Let us compare this to the constrained formulation~\eqref{prob:nonconvex_constr}:
First, we take into account that the inequality constraint \(\norm{AZ-Y}_\fro \leq \delta\) is generally fulfilled with equality at optimal solutions, since getting a better fit usually requires increasing the regularization term.
This is indeed the case for \(\gamma > 0\), but not necessarily for \(\gamma = 0\),
and we will discuss this further below.
Second, \(1/(2\alpha)\) can be interpreted as a Lagrange multiplier for this constraint.
Thus, the difference of~\eqref{eq:discrepancy} consists in a fuzzy version of the hard constraint
\(\norm{AZ-Y}_\fro = \delta\), which is known not to affect recovery guarantees for convex problems
since usually \(\delta\) is not known exactly.
For efficiency purposes, experience from convex regularization suggests to initialize \(\alpha\) large (to obtain a highly sparse approximation initially), and only decrease it if~\eqref{eq:discrepancy} is violated, and we employ this as well.
We leave a detailed analysis and algorithmic refinements to future work.

\subsubsection{Discrepancy principle for the \(\owl\) functional}

The choice \(\gamma = 0\) does not always fit to the discrepancy principle~\eqref{eq:discrepancy}, and in some cases we can only enforce the upper bound.
This is due to the fact that unless \(\gamma > 0\), a change in the variable \(Z\) that is allowed to increase the discrepancy does not necessarily allow to reduce the regularization term \(\Psi_\gamma\) by moving the reconstruction closer to zero.
Note that
\[
  \Psi_\gamma(\tau Z) \leq \Psi_\gamma(Z)
  \quad\text{for } 0 \leq \tau < 1
\]
with equality for \(\gamma = 0\) and strict inequality for \(\gamma > 0\) and \(Z\neq 0\).
The latter property allows to directly show that any solution of~\eqref{prob:nonconvex_constr} will always fulfill \(\norm{AZ-Y}_\fro = \delta\) for \(\gamma > 0\).
  
To illustrate the issue for \(\gamma = 0\) in more detail, consider the case of matrices with \(\rank(Z) = \ellzero(Z) = K\), and the optimality conditions from Theorem~\ref{thm:opt_cond_model}.
In that case, for the set of vectors \(z_n\) with \(z_n \neq 0\) the vectors \(u_n = W^{-1/2} z_n = (Z^TZ)^{-1/2} z_n\) form an orthogonal basis of \(\R^K\) and
\[
  \frac{1}{\norm{z_n}_{W}} W
  - \sum_{i\colon z_i \neq 0} \frac{1}{\norm{z_i}_{W}} W z_i z^\top_i W
  =  W^{1/2} \left(\frac{1}{\norm{u_n}_2} I - \sum_{i\colon z_i \neq 0} \frac{1}{\norm{u_i}_{2}} u_i u^\top_i\right) W^{1/2}
  = 0.
\]
Thus, by the first condition in Theorem~\ref{thm:opt_cond_model}, any optimal solution of~\eqref{prob:nonconvex_reg} fulfills \((A^T(A \bar{Z} - Y))_n = 0\) for all \(n\) with \(z_n \neq 0\).
For \(Z\) restricted to its sparsity pattern \(I(Z) = \{\,i\;|\; z_i \neq 0\,\}\), \(\abs{I(Z)} = K\), such that \(Z = P_{I(Z)} Z'\) with \(Z' \in \R^{K\times K}\) of full rank and \(P_{I(Z)} \in \{0,1\}^{N\times K}\) a selection of columns of the identity matrix, we can suppose that \(P_{I(Z)}^T A^T A P_{I(Z)}\) is invertible (which is necessary for recovery).
The set of equations derived above can be written as \(P_{I(Z)}^TA^T(A P_{I(Z)} Z' - Y) = 0\), which can always be uniquely solved.

In particular, consider a situation with \(K\)-sparse full rank {\(X = P_{I(X)}X'\)} and noisy data with \(\delta_X = \norm{AX-Y}\).
Now, consider that part of the noise can be fitted on the same support as \(X\), and that in general there exists a
\(K\)-sparse full rank \(Z'\) such that {\(\delta' = \norm{AP_{I(X)}Z'-Y}_\fro < \delta_X\).}
In fact, let \(Z'\) be the minimizer of the discrepancy \(\delta'\) on the sparsity pattern of \(X\).
Note that this minimizer exactly fulfills  {\(P_{I(X)}^TA^T(A P_{I(X)} Z' - Y) = 0\).}
Then,  {\(\bar{Z} = P_{I(X)}Z'\)} clearly is a solution of~\eqref{prob:nonconvex_constr} for all \(\delta > \delta'\).
Let \(\alpha'\) be the maximum of  {\(\norm{(A^T (A P_{I(X)} Z' - Y))_n}_{Z^TZ}\)} over all \(n = 1,2\ldots,N\).
Clearly, for all \(\alpha \geq \alpha'\), the  {\(\bar{Z} = P_{I(X)}Z'\)} also fulfills the optimality conditions from Theorem~\ref{thm:opt_cond_model}.
In this situation, Algorithm~\ref{alg:practical}, which exclusively operates on rank-\(K\) matrices for \(\gamma = 0\), will tend to converge towards  {\(\bar{Z} = P_{I(X)}Z'\)}, regardless of the value of \(\alpha \geq \alpha'\), since the only way to decrease the regularization term below \(\Psi(\bar{Z}) = K\) is to decrease the rank.
Here, if \(\delta' \ll \delta_X\), the lower bound in condition~\eqref{eq:discrepancy} can not be satisfied by increasing \(\alpha\). In fact, such a situation always occurs for \(M = K\), where  {\(A P_{I(X)} Z' = Y\)} can be uniquely solved for \(Z'\) and \(\delta' = \alpha' = 0\).

To guard against this situation, we add an additional rule that allows termination of the algorithm if the upper bound in~\eqref{eq:discrepancy} is valid, \(\alpha\) has been previously increased, and no change in \(Z_\alpha\) has occurred.

\subsection{Continuation for relaxation parameter}
In the previous parts of this section \(\gamma\) was considered fixed.
Applying the described algorithms directly for \(\gamma=0\), and solving the original problem~\eqref{prob:nonconvex_reg} has several down-sides:
\begin{itemize}
\item
  Due to the high degree of nonconvexity of the regularizer, the solution will depend on initialization.
  This, aside from causing problems with identifying a global minimum, also requires the initial point to be of full rank, which is not efficient when \(K\) is large).
\item
  If the optimal solution has \(\rank{\bar{Z}} < K\), the method will not be able to identify it for \(\gamma = 0\), since Algorithm~\ref{alg:practical} produces full rank iterates only and the \(\owl\)-functional is discontinuous across different ranks.
  This has to be remedied by guessing the appropriate rank ahead of time, and reducing \(K\) to this rank by a data reduction as discussed at the beginning of section~\ref{sec:approx}.
\end{itemize}

In this section, we briefly sketch a strategy that starts with \(\gamma=1\) and then successively reduces \(\gamma\).
Here, where we are initially solving the convex problem with \(\owl = \Psi_0\) replaced by \(\ellone = \Psi_1\) and then following the path of solutions \(\bar{Z}_{\alpha,\gamma}\) for \(\Psi_\gamma\) to finally obtain \(\bar{Z}_{\alpha_\delta,0}\) with \(\alpha_\delta\) fulfilling~\eqref{eq:discrepancy}.

As a proof of concept, we employ a simple strategy, employing a fixed sequence of values \(\gamma_l = \tilde{\gamma}^l\) for some \(0 < \tilde{\gamma} < 1\) and \(l=0,\ldots,l_{\max}\).
We start with \(l=0\) and use the method from section~\ref{sec:adapt_alpha} to solve the convex problem and the discrepancy principle.
Then, we increment \(l\), and continue iterating Algorithm~\ref{alg:practical} until the termination conditions from the previous section are fulfilled again.
We leave refinements of this naive approach (cf., e.g., \cite{allgower2012numerical}) to future work.

\section{Experiments}\label{sec:experiments}

\subsection{Synthetic data}

We evaluate the performance of our method against other recent MMV techniques, including  subspace-augmented MUSIC (SA-MUSIC) \cite{6158602}, simultaneous normalized iterative hard thresholding (SNIHT) \cite{6719509}, and rank-aware orthogonal recursive matching pursuit (RA-ORMP) \cite{davies2012rank}. The codes of these baseline methods were obtained from the repository \href{https://github.com/AdriBesson/joint_sparse_algorithms}{https://github.com/AdriBesson/joint\_sparse\_algorithms}. 
We also include in our evaluation the \texttt{spg\_mmv} method, which solves the standard $\ell_{2,1}$ problem in constrained form and whose codes are acquired from SPGL1 package \cite{van2009probing, doi:10.1137/100785028}. Our codes and data for reproducing $\owl$ results are available at  \href{https://github.com/a-petr/OWL}{https://github.com/a-petr/OWL}.

We use similar experiment setting as those in \cite{8531673}. We consider a Gaussian random measurement matrix $ A\in \mathbb{R}^{M\times N}$, with ${A}_{ij} \sim \mathcal{N}(0,1/\sqrt{M})$ and $N$ being fixed to $128$. The signal matrix ${X} \in \mathbb{R}^{N\times K}$ is built as a random matrix, with $K = s = 30$, where $s$ is the row sparsity. We conduct $40$ random trials of the algorithms for each experiment and compute the rate of successful support recovery. We compare our method with their counterpart in two tests: fixed rank ($r = 10$) for a number of measurements ranging between $30$ and $90$ and fixed number of measurements ($M = 51$) for a rank varying between $1$ and $30$.

We first consider the noiseless scenario, i.e., observation $ Y =  A  X$. In Figure~\ref{fig:SR_noiseless}, our proposed approach is second to RA-ORMP in recovery probability  in both fixed number of measurements (left) and fixed rank (right) tests, followed by SA-MUSIC and SNIHT. The experiment shows that these four methods are rank-aware, because increasing the rank of the signal matrix leads to an improvement of the recovery rate, while $\ell_{2,1}$ regularization with \texttt{spg\_mmv} solver does not benefit from the change of the rank. For $\owl$ regularization, the recovery rate grows fast when the rank of signal matrix is larger than 3 and gets to $100\%$ at rank $12$.  SA-MUSIC and SNIHT struggle to identify the correct support of the solutions, but when they are able to do so, they achieve a reconstruction error smaller than that of $\owl$ (Figure~\ref{fig:error_noiseless}).

\begin{figure}[h!]
    \centering
        \includegraphics[width=0.45\textwidth]{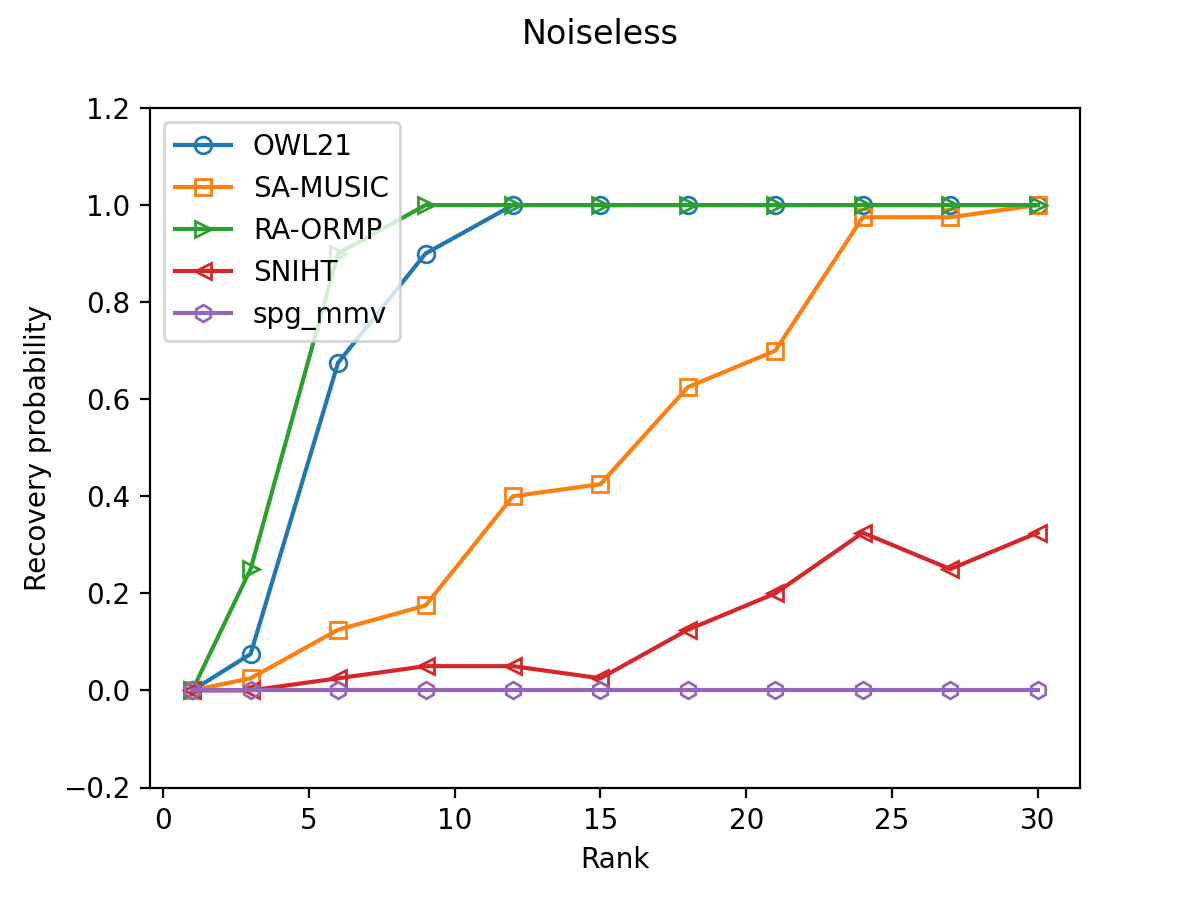}
        \includegraphics[width=0.45\textwidth]{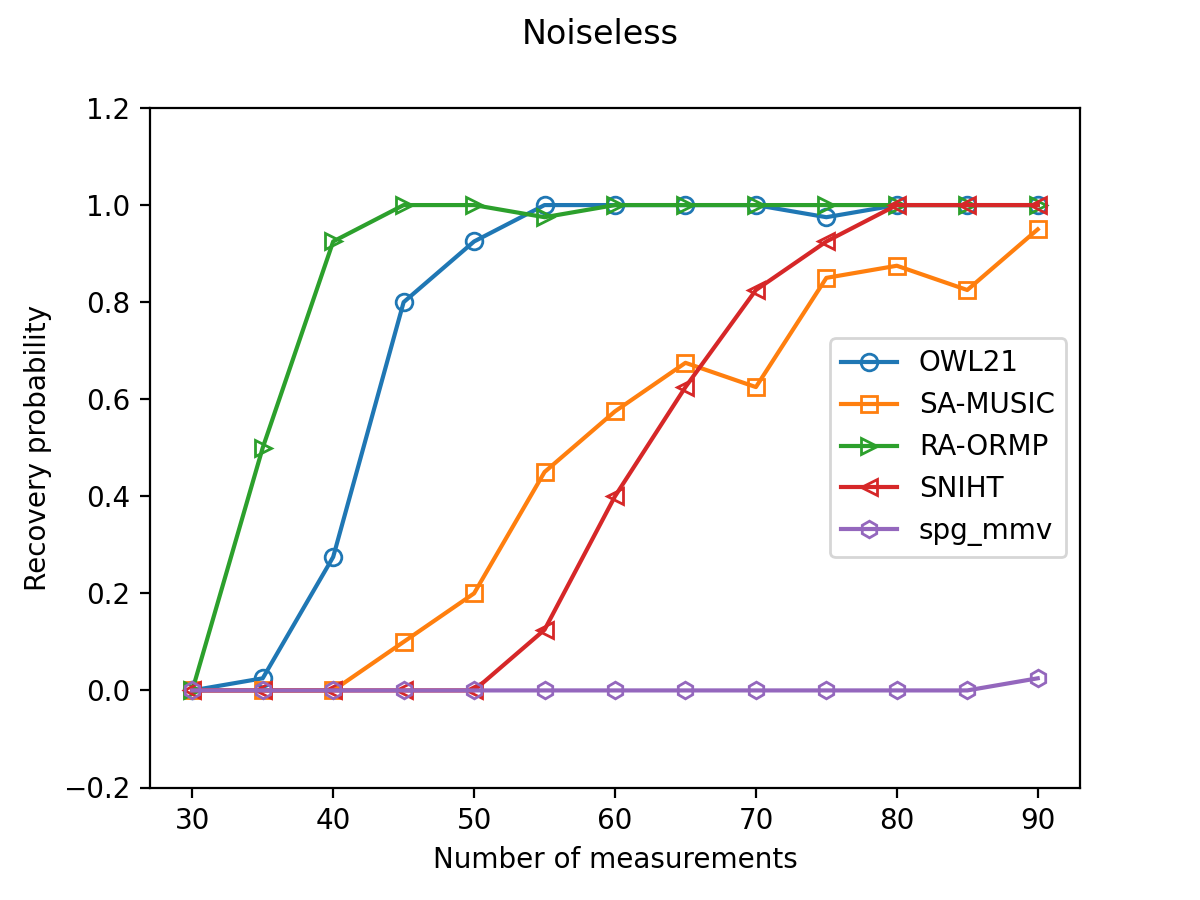}
        \caption{Recovery probability of our approach in comparison with other joint sparse recovery techniques for varying matrix rank (left) and varying number of measurements (right) for noiseless reconstruction.}
        \label{fig:SR_noiseless}
\end{figure}

\begin{figure}[h!]
    \centering
        \includegraphics[width=0.45\textwidth]{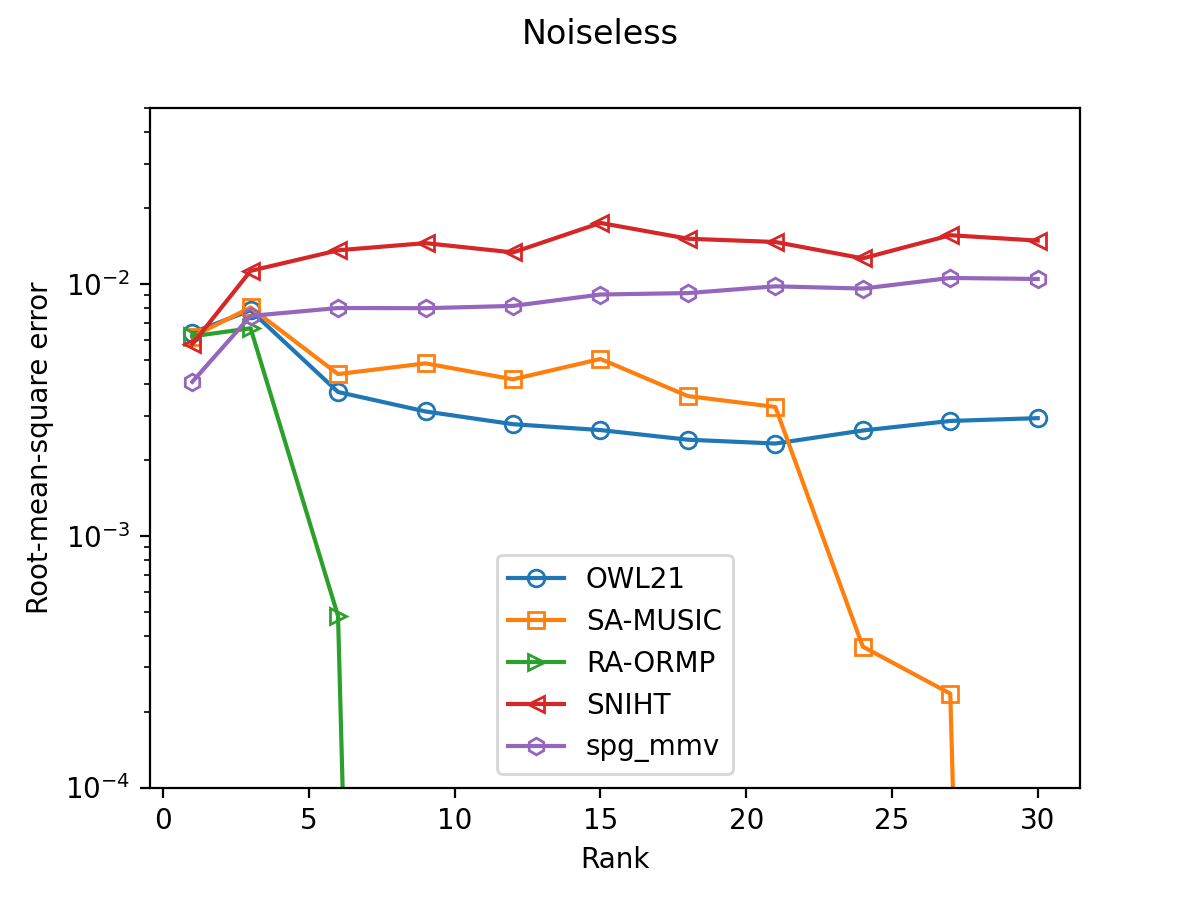}
        \includegraphics[width=0.45\textwidth]{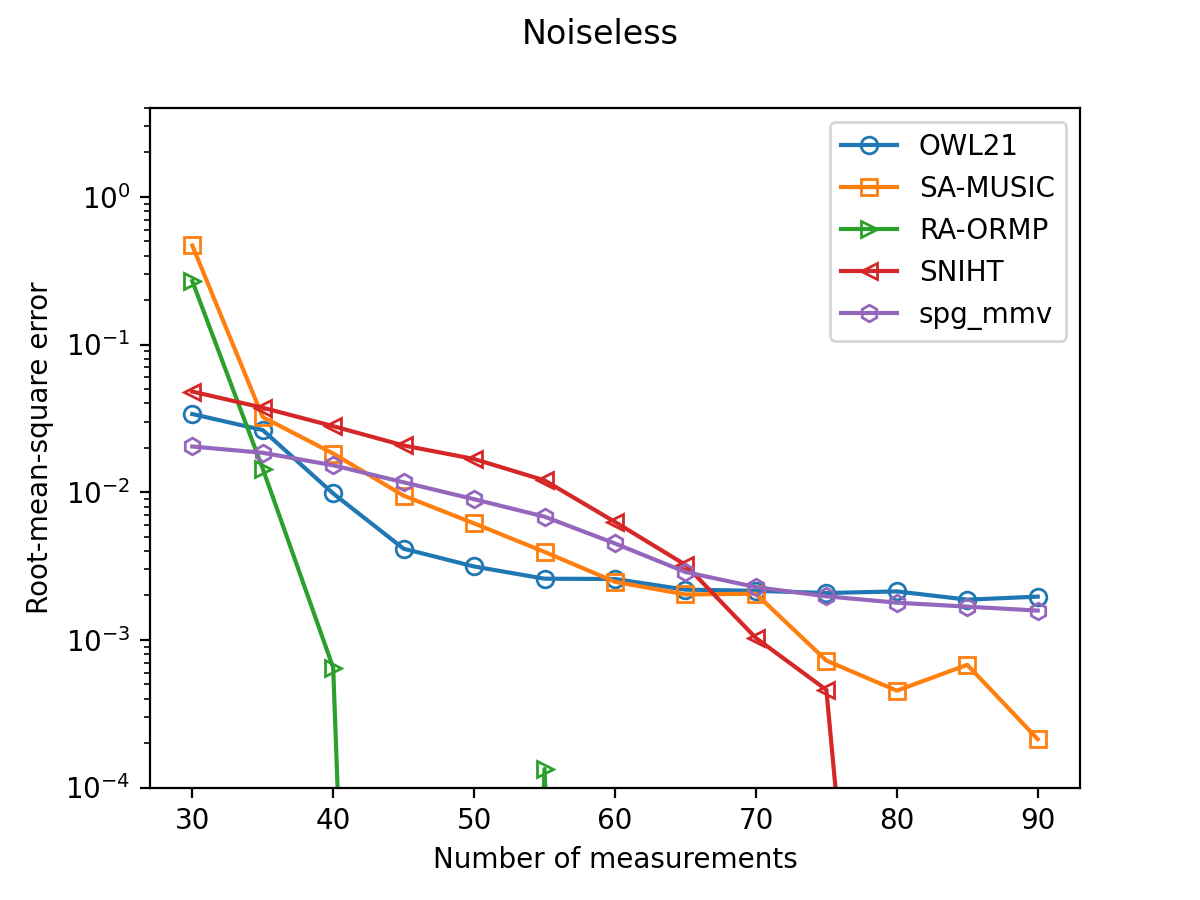}
        \caption{Root-mean-square error of our approach in comparison with other joint sparse recovery techniques for varying matrix rank (left) and varying number of measurements (right) for noiseless reconstruction.}
        \label{fig:error_noiseless}
\end{figure}

Next, we consider a more practical scenario where the observation $Y$ contains some measurement noise, i.e., $ Y =  A  X + \eta$. We assume $ \eta$ is a normal distributed random variable with standard deviation ${0.1}/{\sqrt{MK}}$, so that the expectation of $\| \eta\|$ is $0.1$. In Figure~\ref{fig:SR_noisy}, we observe that $\owl$ achieves best performance in recovery rate in both fixed number of measurements (left) and fixed rank (right) tests. Our method has $100\%$ recovery rate at signal rank $18$, while SA-MUSIC only achieves this at full signal rank and SNIHT cannot get a recovery rate higher than $40\%$. Interestingly, the best overall baseline in noiseless scenario, RA-ORMP, suffers heavily from the presence of the noise. With fixed rank and varied number of measurements, $\owl$ also outperforms all others by a substantial margin. Our method has the best performance in root-mean-square error in the small rank and small number of measurements scenarios, where its recovery rate is superior (Figure~\ref{fig:error_noisy}). However, when the other methods catch up and can approximate the signal support well, they provide better reconstruction error, similarly to what we have seen in the noiseless case.

To conclude, the numerical tests highlight that $\owl$ can efficiently exploit the rank of the signal matrix and achieve a very competitive recovery rate, with mild requirements on the number of measurements and signal rank. They also reveal that the successful rate and signal error metrics may not always be in alignment, thus, one should consider the recovery goals, measurement budget, and signal noise to select a suitable method. Our approach excels in identifying the exact support of the row-sparse signals, and has the top performance when the signal matrices have low-rank, making it a strong choice when those are the priority. 

\begin{figure}[h!]
    \centering
        \includegraphics[width=0.45\textwidth]{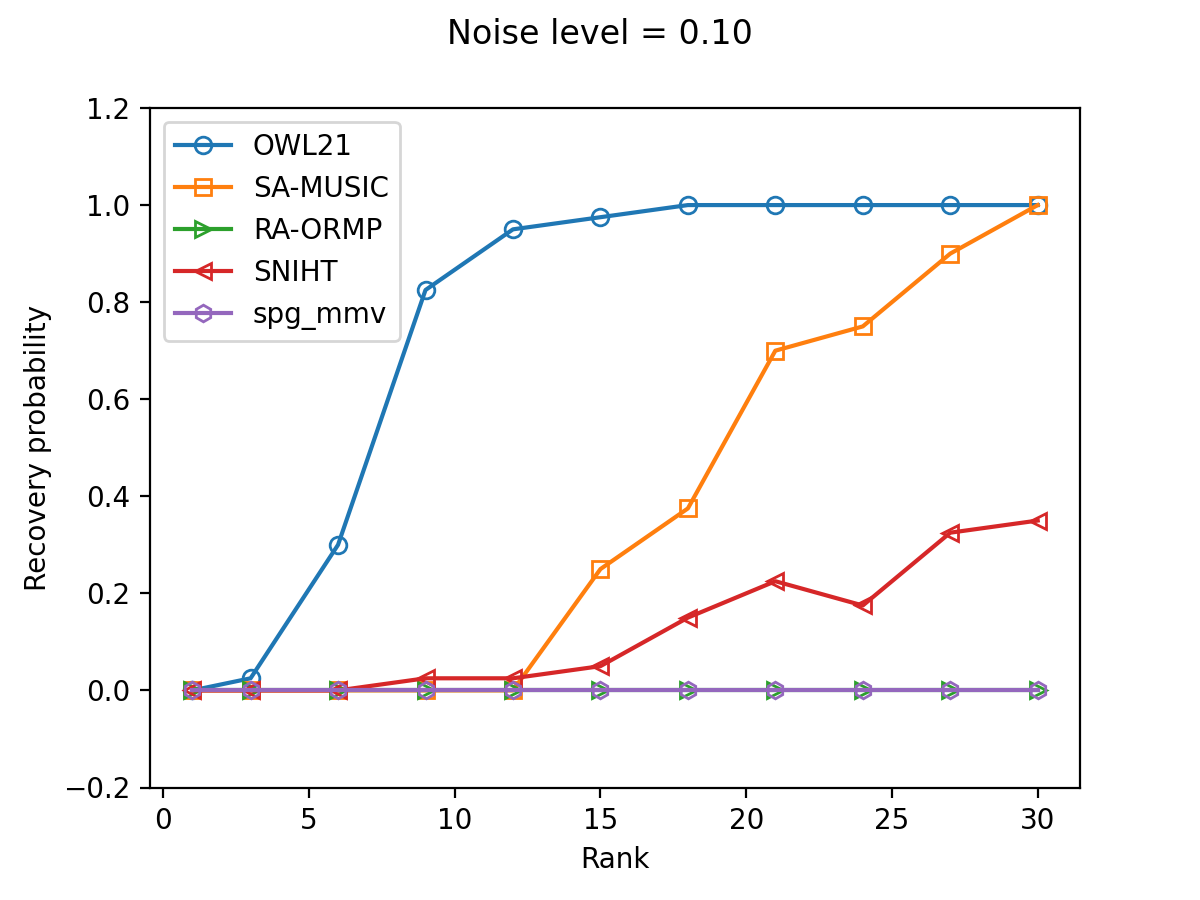}
        \includegraphics[width=0.45\textwidth]{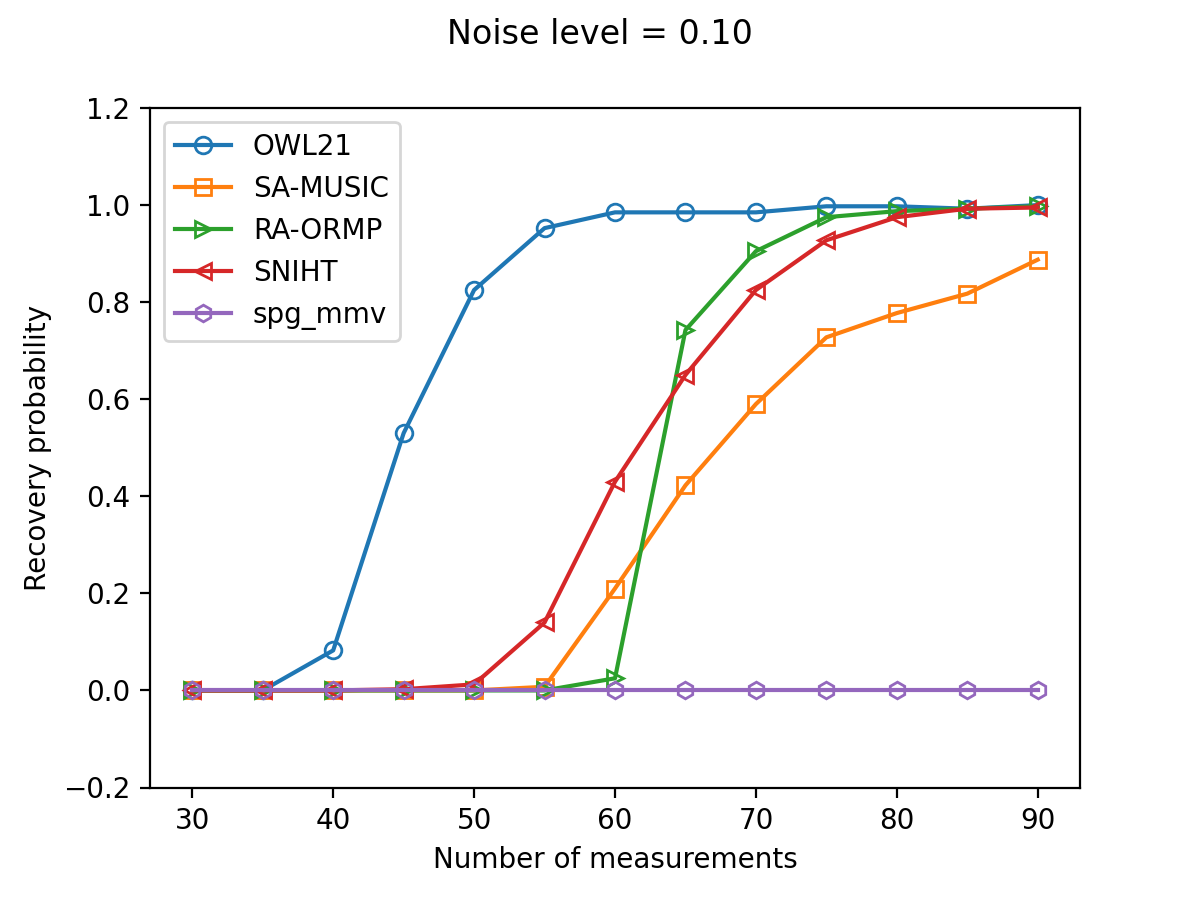}
        \caption{Recovery probability of our approach in comparison with other joint sparse recovery techniques for varying matrix rank (left) and varying number of measurements (right) for noisy reconstruction.}
        \label{fig:SR_noisy}
\end{figure}

\begin{figure}[h!]
    \centering
        \includegraphics[width=0.45\textwidth]{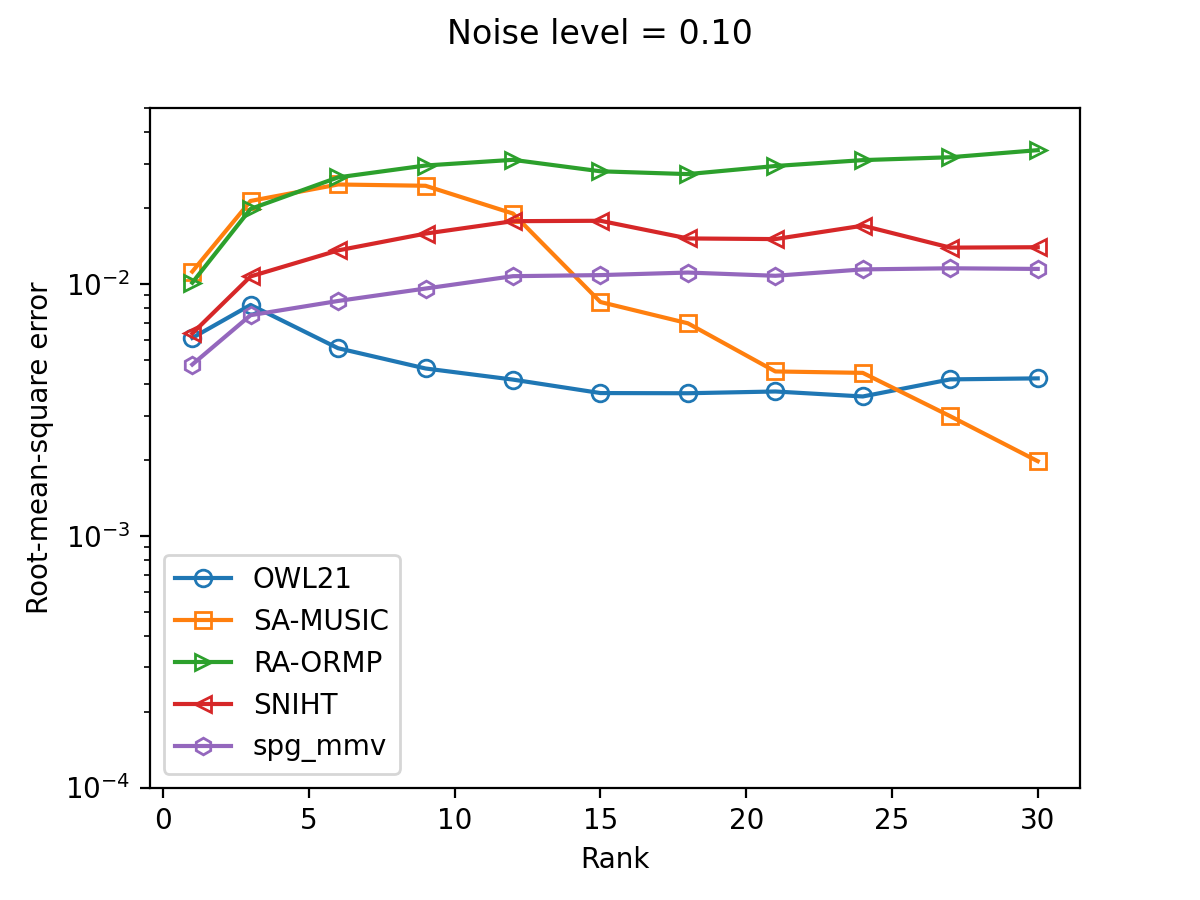}
        \includegraphics[width=0.45\textwidth]{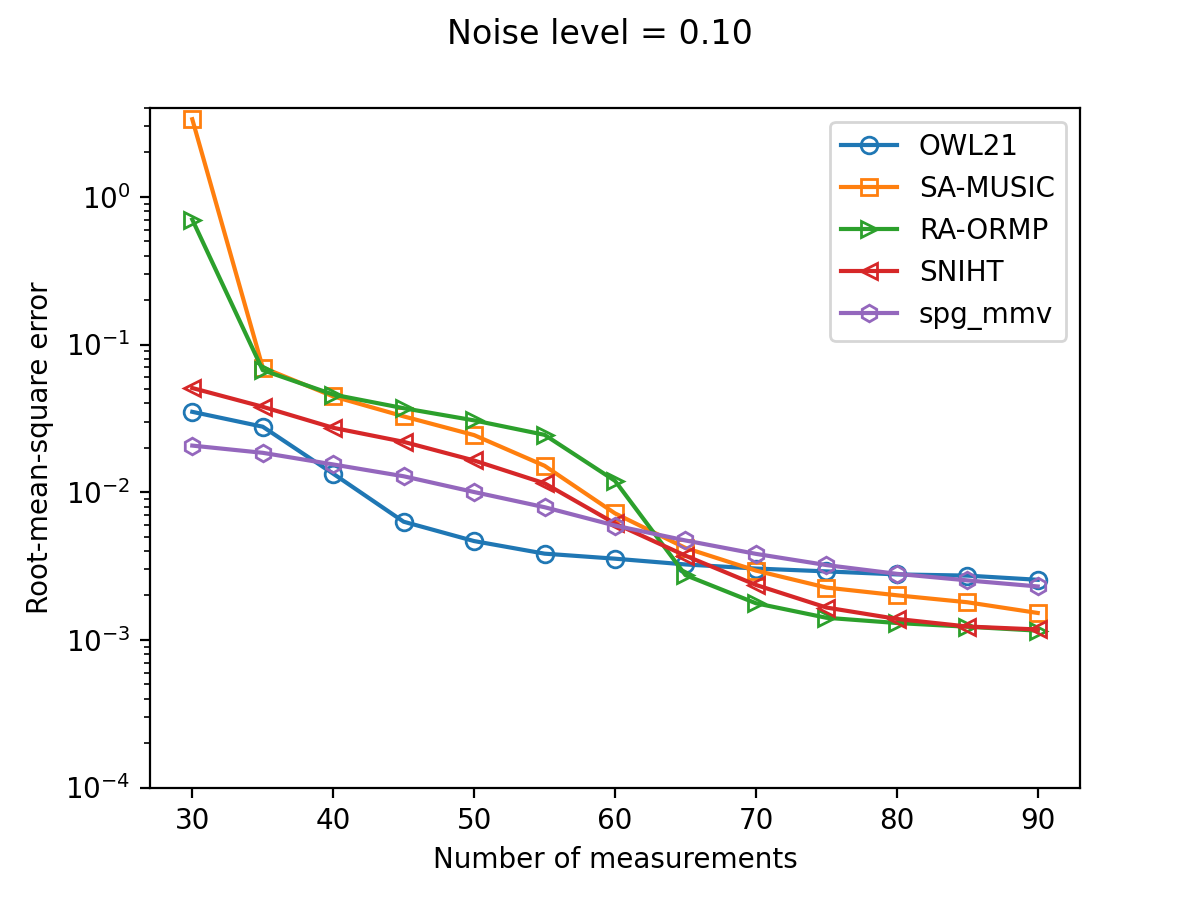}
        \caption{Root-mean-square error of our approach in comparison with other joint sparse recovery techniques for varying matrix rank (left) and varying number of measurements (right) for noisy reconstruction.}
        \label{fig:error_noisy}
\end{figure}

\subsection{Feature selection on biological data.}
\label{sec:feature_selection}
We evaluate the performance of our method on the feature selection problem for two publicly available biomedical informatics data sets. The first data set WISCONSIN is a collection of data about breast cancer tumors that were diagnosed as either malignant or benign \cite{Bennett1992RobustLP}. The data set contains 714 samples, with each sample representing a breast cancer tumor. Each tumor is described by 30 features, including information about the size, shape, and texture of the tumor, etc. These features were computed from a digitized image of a fine needle aspirate of a breast mass. The second data set LUNG\_DISCRETE is a collection of microarray data about lung carcinomas, containing $147$ samples in $7$ classes where each sample consists of 325 gene expressions \cite{1453511}. In the feature selection problem, we aim to find an optimal subset of features (gene expressions) that can accurately represent the whole data set. This problem is a special instance of the MMV problem, where the observation matrix $Y$ is identical with the measurement matrix $A$ (see Section~\ref{sec:related}). Here, we apply $\owl$ and compare it with other MMV techniques (SA-MUSIC, RA-ORMP and \texttt{spg\_mmv}) in reconstructing the full data set, given different allowable number of features.  SNIHT does not produce competitive results for this test, so will be omitted from the comparison. For SA-MUSIC and RA-ORMP, the allowable number of features is enforced directly, while for the two regularization-based methods, we set an error tolerance a priori and solve for the smallest set of features that can reconstruct the full data set within this tolerance. We rank the feature importance via the $\ell_2$ norm of the respective row of the solution matrix and further prune out the spurious features whose importance scores are very small, in particular, less than $10^{-6}\times\text{(maximum importance score)}$. 
Finally, after the features are selected, we reconstruct the full data from this set by solving the regular least square problem, and report the root-mean-square error of the approximated and ground truth data.

Note that the feature selection problem has $K=N$ thus often involves a large $K$. To accelerate the $K\times K$ matrix inversion step in our algorithm, we employ an SVD technique to decompose \(A = P\Sigma Q\) and solve the reduced problem 
\begin{align*}
\min_{Z\in \R^{N\times K'} } \left(\owl(Z)+\frac{1}{2\alpha} \|AZ-A'\|_\fro^2\right).
\end{align*} 
Here, $A' = P\Sigma' \in \R^{M\times K'}$ is a low rank approximation of $P\Sigma$, formed by dropping the columns of $\Sigma$ corresponding to singular values that are below a specified threshold. The solution now has dimension $N\times K'$, and our algorithm involves the matrix inversion of size $K'\times K'$ instead, which is much faster if \(K' \ll K\). In the LUNG\_DISCRETE problem, even though $A$ has dimension ${147\times 325}$ (i.e., $K=325$), its effective rank is approximately $72$ (there are only 72 singular values of $A$ exceeding $10^{-12}$), thus $K' =72$. We emphasize that this is only the most conservative choice, and a more aggressive thresholding strategy is possible. In fact, here we further reduce $K'$ by linking the singular value cutoff with the predefined reconstruction error tolerance, as the low rank approximation $A'$ can afford an error of the same scale as this tolerance. The final row-sparse solution is easily obtained by right multiplying the reduced solution with $Q$. It is worth noting that the above strategy is general for feature selection problems because the data matrices of these problems are inherently low-rank.

For the WISCONSIN data set (Figure~\ref{fig:error_bio}, left), we observe that $\owl$ performs much better than the competitors in identifying the representative features. The reconstruction error of $\owl$ drops rapidly when the number of features is small and is approximately $20\%$ lower than those of the next baselines. By contrast, the other methods generally require five more features to capture the data set as good as $\owl$. The performance of $\owl$ and SA-MUSIC eventually converge around 25 features, where the reconstruction error diminishes to $0$. This verifies that  the feature set contains some degree of redundancy, and that the data set can still be accurately represented with some features being removed. For the LUNG\_DISCRETE data set (Figure~\ref{fig:error_bio}, right), the performances of $\owl$, RA-ORMP and SA-MUSIC are close and better than \texttt{spg-mmv} with small number of features. However, all the tested methods can select roughly $72$ optimal features that fully represent the data, despite the extensive feature set comprising 325 gene expressions. 

\begin{figure}[h!]
    \centering
        \includegraphics[width=0.45\textwidth]{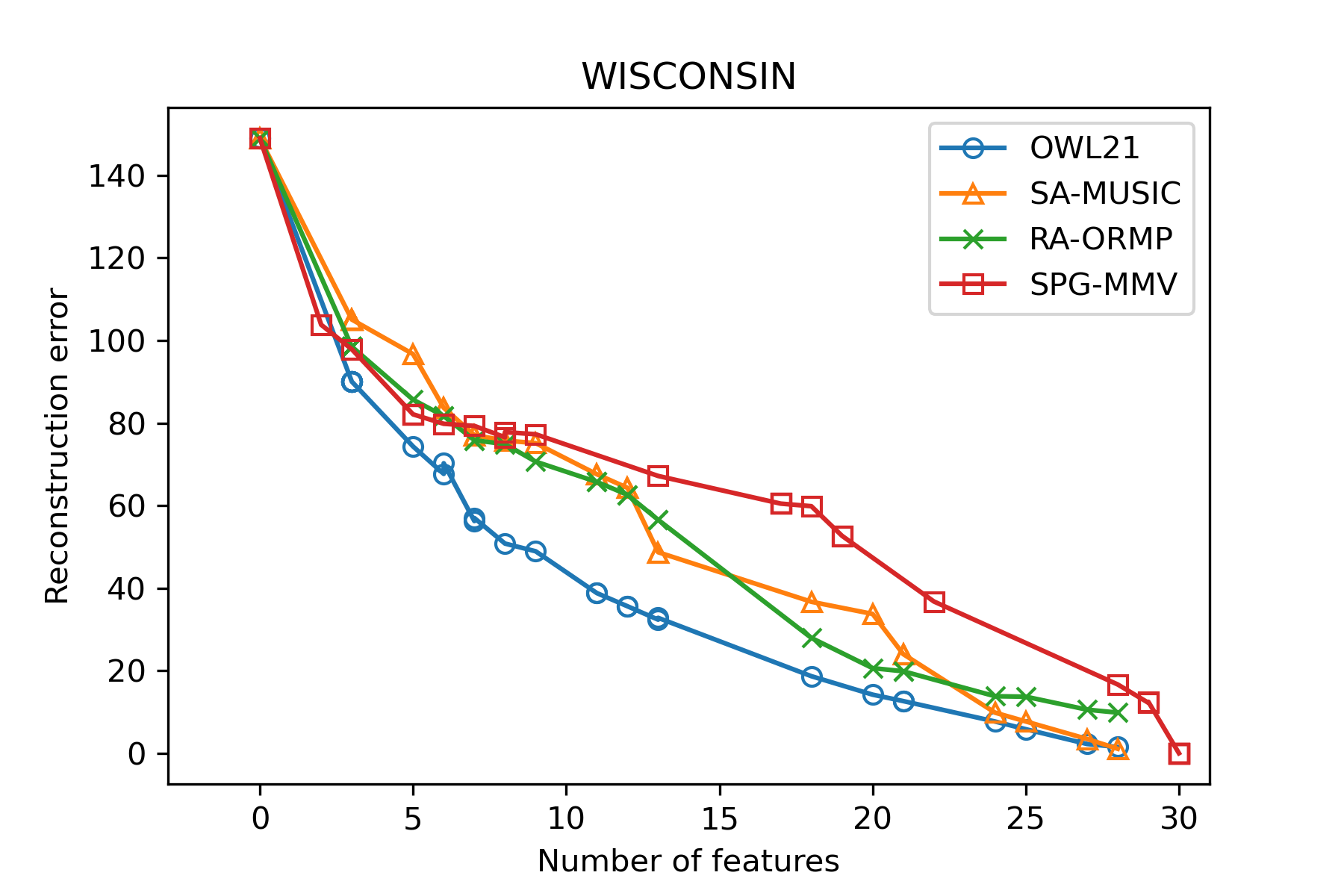}
        \includegraphics[width=0.45\textwidth]{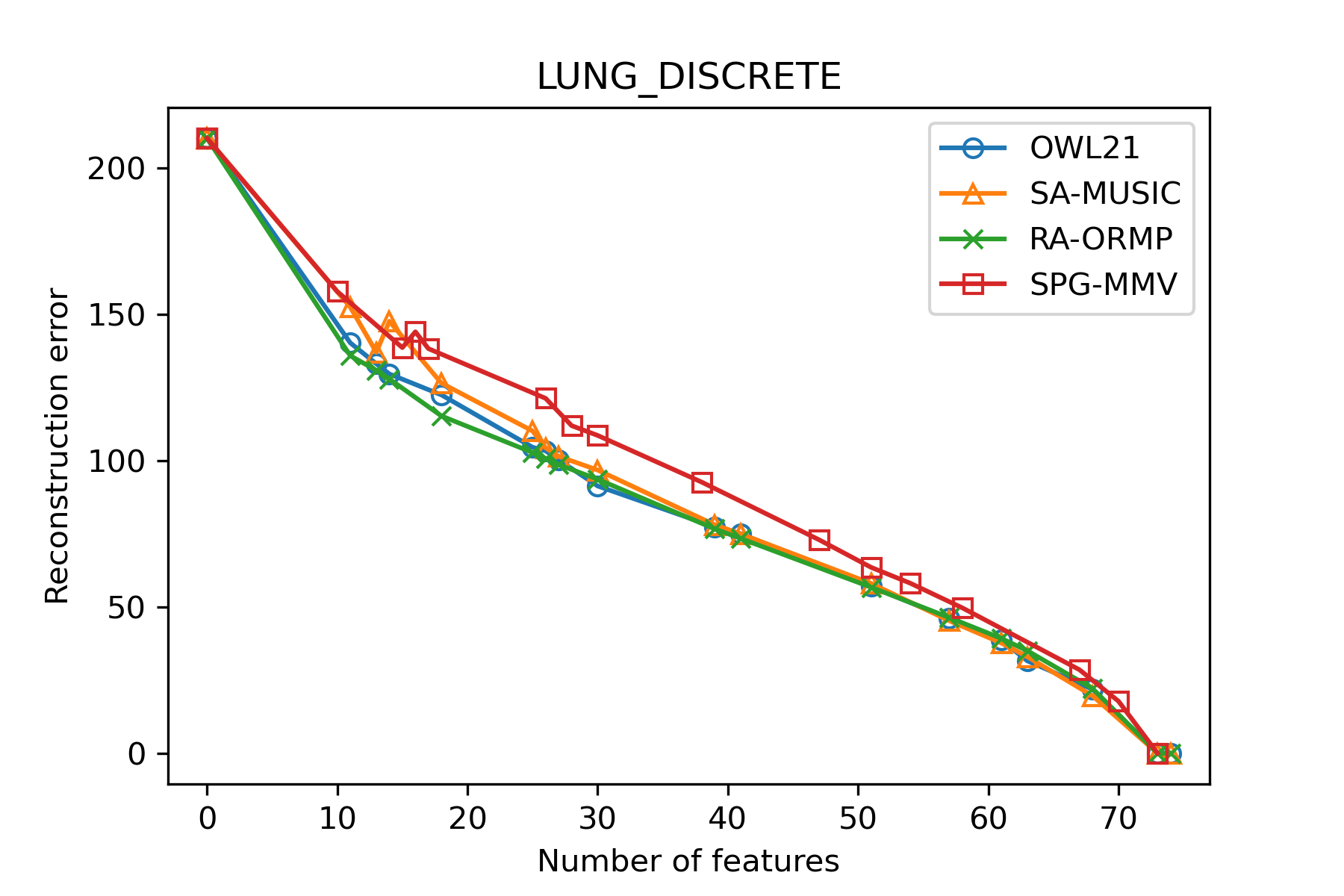}
        \caption{Approximation of the full WISCONSIN (left) and LUNG\_DISCRETE datasets (right) from a limited number of features, identified by $\owl$ as opposed to other methods, when a varying number of features is selected.}
        \label{fig:error_bio}
\end{figure}

\bibliography{refs.bib}

\begin{thebibliography}{10}

\bibitem{allgower2012numerical}
{\sc E.~L. Allgower and K.~Georg}, {\em Numerical continuation methods: an
  introduction}, vol.~13, Springer Science \& Business Media, 2012.

\bibitem{AnderssonCarlssonPerfekt2015}
{\sc F.~Andersson, M.~Carlsson, and K.-M. Perfekt}, {\em Operator-{Lipschitz}
  estimates for the singular value functional calculus}, Proceedings of the
  American Mathematical Society, 144 (2015), p.~1867–1875.

\bibitem{doi:10.1137/080716542}
{\sc A.~Beck and M.~Teboulle}, {\em A fast iterative shrinkage-thresholding
  algorithm for linear inverse problems}, SIAM Journal on Imaging Sciences, 2
  (2009), pp.~183--202.

\bibitem{Bennett1992RobustLP}
{\sc K.~P. Bennett and O.~L. Mangasarian}, {\em Robust linear programming
  discrimination of two linearly inseparable sets}, Optimization Methods \&
  Software, 1 (1992), pp.~23--34.

\bibitem{8531673}
{\sc A.~Besson, D.~Perdios, Y.~Wiaux, and J.-P. Thiran}, {\em Joint sparsity
  with partially known support and application to ultrasound imaging}, IEEE
  Signal Processing Letters, 26 (2019), pp.~84--88.

\bibitem{6719509}
{\sc J.~D. Blanchard, M.~Cermak, D.~Hanle, and Y.~Jing}, {\em Greedy algorithms
  for joint sparse recovery}, IEEE Transactions on Signal Processing, 62
  (2014), pp.~1694--1704.

\bibitem{blanchard2020rank}
{\sc J.~D. Blanchard, C.~Leedy, and Y.~Wu}, {\em On rank awareness,
  thresholding, and {MUSIC} for joint sparse recovery}, Applied and
  Computational Harmonic Analysis, 48 (2020), pp.~482--495.

\bibitem{cai2021robust}
{\sc H.~Cai, K.~Hamm, L.~Huang, and D.~Needell}, {\em Robust {CUR}
  decomposition: Theory and imaging applications}, SIAM Journal on Imaging
  Sciences, 14 (2021), pp.~1472--1503.

\bibitem{chambolle-pock-2011}
{\sc A.~Chambolle and T.~Pock}, {\em A first-order primal-dual algorithm for
  convex problems with applications to imaging}, J Math Imaging Vis, 40 (2011),
  pp.~120--145.

\bibitem{clarke2013functional}
{\sc F.~Clarke}, {\em Functional analysis, calculus of variations and optimal
  control}, vol.~264, Springer, 2013.

\bibitem{doi:10.1080/02331930412331327157}
{\sc P.~L. Combettes}, {\em Solving monotone inclusions via compositions of
  nonexpansive averaged operators}, Optimization, 53 (2004), pp.~475--504.

\bibitem{doi:10.1137/060669498}
{\sc P.~L. Combettes and J.-C. Pesquet}, {\em Proximal thresholding algorithm
  for minimization over orthonormal bases}, SIAM Journal on Optimization, 18
  (2008), pp.~1351--1376.

\bibitem{ConnGouldToint:2000}
{\sc A.~R. Conn, N.~I.~M. Gould, and P.~L. Toint}, {\em Trust Region Methods},
  Society for Industrial and Applied Mathematics, 2000.

\bibitem{1453780}
{\sc S.~Cotter, B.~Rao, K.~Engan, and K.~Kreutz-Delgado}, {\em Sparse solutions
  to linear inverse problems with multiple measurement vectors}, IEEE
  Transactions on Signal Processing, 53 (2005), pp.~2477--2488.

\bibitem{dal2012introduction}
{\sc G.~Dal~Maso}, {\em An introduction to $\Gamma$-convergence}, vol.~8,
  Springer Science \& Business Media, 2012.

\bibitem{davies2012rank}
{\sc M.~E. Davies and Y.~C. Eldar}, {\em Rank awareness in joint sparse
  recovery}, IEEE Transactions on Information Theory, 58 (2012),
  pp.~1135--1146.

\bibitem{refId0}
{\sc N.~Dexter, H.~Tran, and C.~Webster}, {\em A mixed egularization approach
  for sparse simultaneous approximation of parameterized {PDE}s}, ESAIM: M2AN,
  53 (2019), pp.~2025--2045.

\bibitem{Dexter-SVVA22}
\leavevmode\vrule height 2pt depth -1.6pt width 23pt, {\em On the strong
  convergence of forward-backward splitting in reconstructing jointly sparse
  signals}, Set-Valued Var. Anal, 30 (2022), pp.~543--557.

\bibitem{EldarRauhut10}
{\sc Y.~Eldar and H.~Rauhut}, {\em {Average Case Analysis of Multichannel
  Sparse Recovery Using Convex Relaxation}}, IEEE Transactions on Information
  Theory, 56 (2010), pp.~505--519.

\bibitem{Engl2015}
{\sc H.~W. Engl and R.~Ramlau}, {\em Regularization of Inverse Problems},
  Springer Berlin Heidelberg, Berlin, Heidelberg, 2015, pp.~1233--1241.

\bibitem{doi:10.1137/13090540X}
{\sc E.~Esser, Y.~Lou, and J.~Xin}, {\em A method for finding structured sparse
  solutions to nonnegative least squares problems with applications}, SIAM
  Journal on Imaging Sciences, 6 (2013), pp.~2010--2046.

\bibitem{544131}
{\sc P.~Feng and Y.~Bresler}, {\em Spectrum-blind minimum-rate sampling and
  reconstruction of multiband signals}, in 1996 IEEE International Conference
  on Acoustics, Speech, and Signal Processing Conference Proceedings, vol.~3,
  1996, pp.~1688--1691 vol. 3.

\bibitem{FornasierRauhut08}
{\sc M.~Fornasier and H.~Rauhut}, {\em {Recovery algorithms for vector-valued
  data with joint sparsity constraints}}, SIAM Journal on Numerical Analysis,
  46 (2008), pp.~577--613.

\bibitem{6288484}
{\sc M.~Golbabaee and P.~Vandergheynst}, {\em Hyperspectral image compressed
  sensing via low-rank and joint-sparse matrix recovery}, in 2012 IEEE
  International Conference on Acoustics, Speech and Signal Processing (ICASSP),
  2012, pp.~2741--2744.

\bibitem{GRSV08}
{\sc R.~Gribonval, H.~Rauhut, K.~Schnass, and P.~Vandergheynst}, {\em {Atoms of
  All Channels, Unite! Average Case Analysis of Multi-Channel Sparse Recovery
  Using Greedy Algorithms}}, {Journal of Fourier Analysis and Applications}, 14
  (2008), pp.~655--687.

\bibitem{hamm2020perspectives}
{\sc K.~Hamm and L.~Huang}, {\em Perspectives on {CUR} decompositions}, Applied
  and Computational Harmonic Analysis, 48 (2020), pp.~1088--1099.

\bibitem{higham2008functions}
{\sc N.~J. Higham}, {\em Functions of matrices: theory and computation}, SIAM,
  2008.

\bibitem{hoyer2004non}
{\sc P.~O. Hoyer}, {\em Non-negative matrix factorization with sparseness
  constraints.}, Journal of machine learning research, 5 (2004).

\bibitem{JI2012295}
{\sc H.~Ji, J.~Li, Z.~Shen, and K.~Wang}, {\em Image deconvolution using a
  characterization of sharp images in wavelet domain}, Applied and
  Computational Harmonic Analysis, 32 (2012), pp.~295--304.

\bibitem{6122004}
{\sc J.~M. Kim, O.~K. Lee, and J.~C. Ye}, {\em Compressive {MUSIC}: Revisiting
  the link between compressive sensing and array signal processing}, IEEE
  Transactions on Information Theory, 58 (2012), pp.~278--301.

\bibitem{5995521}
{\sc D.~Krishnan, T.~Tay, and R.~Fergus}, {\em Blind deconvolution using a
  normalized sparsity measure}, in CVPR 2011, 2011, pp.~233--240.

\bibitem{lai2011null}
{\sc M.-J. Lai and Y.~Liu}, {\em The null space property for sparse recovery
  from multiple measurement vectors}, Applied and Computational Harmonic
  Analysis, 30 (2011), pp.~402--406.

\bibitem{6158602}
{\sc K.~Lee, Y.~Bresler, and M.~Junge}, {\em Subspace methods for joint sparse
  recovery}, IEEE Transactions on Information Theory, 58 (2012),
  pp.~3613--3641.

\bibitem{10.5555/1795114.1795154}
{\sc J.~Liu, S.~Ji, and J.~Ye}, {\em Multi-task feature learning via efficient
  $l_{2,1}$-norm minimization}, in Proceedings of the Twenty-Fifth Conference
  on Uncertainty in Artificial Intelligence, UAI '09, Arlington, Virginia, USA,
  2009, AUAI Press, p.~339–348.

\bibitem{lu2019embedded}
{\sc M.~Lu}, {\em Embedded feature selection accounting for unknown data
  heterogeneity}, Expert Systems with Applications, 119 (2019), pp.~350--361.

\bibitem{mishali2008reduce}
{\sc M.~Mishali and Y.~C. Eldar}, {\em Reduce and boost: Recovering arbitrary
  sets of jointly sparse vectors}, IEEE Transactions on Signal Processing, 56
  (2008), pp.~4692--4702.

\bibitem{nardone_SMBA}
{\sc D.~Nardone, A.~Ciaramella, and A.~Staiano}, {\em A sparse-modeling based
  approach for class specific feature selection}, PeerJ Computer Science, 5
  (2019), p.~e237.

\bibitem{nie2010efficient}
{\sc F.~Nie, H.~Huang, X.~Cai, and C.~Ding}, {\em Efficient and robust feature
  selection via joint $\ell_{2, 1}$-norms minimization}, Advances in neural
  information processing systems, 23 (2010).

\bibitem{Noll:2010}
{\sc D.~Noll}, {\em Cutting plane oracles to minimize non-smooth non-convex
  functions}, Set-Valued and Variational Analysis, 18 (2010), pp.~531--568.

\bibitem{NollProtRondepierre:2008}
{\sc D.~Noll, O.~Prot, and A.~Rondepierre}, {\em A proximity control algorithm
  to minimize nonsmooth and nonconvex functions}, Pacific Journal of
  Optimization, 4 (2008), pp.~569--602.

\bibitem{1453511}
{\sc H.~Peng, F.~Long, and C.~Ding}, {\em Feature selection based on mutual
  information criteria of max-dependency, max-relevance, and min-redundancy},
  IEEE Transactions on Pattern Analysis and Machine Intelligence, 27 (2005),
  pp.~1226--1238.

\bibitem{petrosyan2019reconstruction}
{\sc A.~Petrosyan, H.~Tran, and C.~Webster}, {\em Reconstruction of jointly
  sparse vectors via manifold optimization}, Applied Numerical Mathematics, 144
  (2019), pp.~140--150.

\bibitem{rahimi2019scale}
{\sc Y.~Rahimi, C.~Wang, H.~Dong, and Y.~Lou}, {\em A scale-invariant approach
  for sparse signal recovery}, SIAM Journal on Scientific Computing, 41 (2019),
  pp.~A3649--A3672.

\bibitem{Schaeffer2017LearningDS}
{\sc H.~Schaeffer, G.~Tran, and R.~A. Ward}, {\em Learning dynamical systems
  and bifurcation via group sparsity}, arXiv: Numerical Analysis,  (2017).

\bibitem{tan2014joint}
{\sc Z.~Tan, P.~Yang, and A.~Nehorai}, {\em Joint sparse recovery method for
  compressed sensing with structured dictionary mismatches}, IEEE Transactions
  on Signal Processing, 62 (2014), pp.~4997--5008.

\bibitem{doi:10.1137/20M136801X}
{\sc M.~Tao}, {\em Minimization of $l_1$ over $l_2$ for sparse signal recovery
  with convergence guarantee}, SIAM Journal on Scientific Computing, 44 (2022),
  pp.~A770--A797.

\bibitem{TROPP2006589}
{\sc J.~A. Tropp}, {\em Algorithms for simultaneous sparse approximation. part
  ii: Convex relaxation}, Signal Processing, 86 (2006), pp.~589--602.
\newblock Sparse Approximations in Signal and Image Processing.

\bibitem{TROPP2006572}
{\sc J.~A. Tropp, A.~C. Gilbert, and M.~J. Strauss}, {\em Algorithms for
  simultaneous sparse approximation. part i: Greedy pursuit}, Signal
  Processing, 86 (2006), pp.~572--588.
\newblock Sparse Approximations in Signal and Image Processing.

\bibitem{van2009probing}
{\sc E.~van~den Berg and M.~P. Friedlander}, {\em Probing the pareto frontier
  for basis pursuit solutions}, Siam journal on scientific computing, 31
  (2009), pp.~890--912.

\bibitem{5452189}
\leavevmode\vrule height 2pt depth -1.6pt width 23pt, {\em Theoretical and
  empirical results for recovery from multiple measurements}, IEEE Transactions
  on Information Theory, 56 (2010), pp.~2516--2527.

\bibitem{doi:10.1137/100785028}
\leavevmode\vrule height 2pt depth -1.6pt width 23pt, {\em Sparse optimization
  with least-squares constraints}, SIAM Journal on Optimization, 21 (2011),
  pp.~1201--1229.

\bibitem{vavasis2009derivation}
{\sc S.~A. Vavasis}, {\em Derivation of compressive sensing theorems from the
  spherical section property}, University of Waterloo, CO, 769 (2009).

\bibitem{9057443}
{\sc C.~Wang, M.~Yan, Y.~Rahimi, and Y.~Lou}, {\em Accelerated schemes for the
  $l_1/l_2$ minimization}, IEEE Transactions on Signal Processing, 68 (2020),
  pp.~2660--2669.

\bibitem{7875091}
{\sc Z.~Wen, B.~Hou, and L.~Jiao}, {\em Joint sparse recovery with
  semisupervised {MUSIC}}, IEEE Signal Processing Letters, 24 (2017),
  pp.~629--633.

\bibitem{XU2021486}
{\sc Y.~Xu, A.~Narayan, H.~Tran, and C.~G. Webster}, {\em Analysis of the ratio
  of $\ell_1$ and $\ell_2$ norms in compressed sensing}, Applied and
  Computational Harmonic Analysis, 55 (2021), pp.~486--511.

\bibitem{Yin2014RatioAD}
{\sc P.~Yin, E.~Esser, and J.~Xin}, {\em Ratio and difference of $\ell_1$ and
  $\ell_2$ norms and sparse representation with coherent dictionaries}, Commun.
  Inf. Syst., 14 (2014), pp.~87--109.

\bibitem{doi:10.1137/20M1355380}
{\sc L.~Zeng, P.~Yu, and T.~K. Pong}, {\em Analysis and algorithms for some
  compressed sensing models based on $\ell_1/\ell_2$ minimization}, SIAM
  Journal on Optimization, 31 (2021), pp.~1576--1603.

\end{thebibliography}
\bibliographystyle{siam}

\newpage
\appendix

\section{Properties of the loss function}\label{sec:rankaware}

\begin{lemma}
Suppose $Z\in S(N,K,r)$ and $Z=U\Sigma V^T$ is the compact singular value decomposition of $Z$, then  
\[\owl(Z)=\|U\|_{2,1}\]
\end{lemma}
\begin{proof}
We have
$(Z^TZ)^{\dagger/2}=V\Sigma^{-1}V^T$,
thus 
$Z(Z^TZ)^{\dagger/2}=UV^T.$ Observe that 
\[
\|Z(Z^TZ)^{\dagger/2}\|_{2,1} = \|UV^T\|_{2,1}=\|U\|_{2,1}.
\qedhere
\]
\end{proof}

\begin{lemma}\label{lem:Psibound}

Let $Z\in \R^{N\times K}$, then \[\rank(Z)\leq \owl(Z)\leq \sqrt{\rank(Z)\cdot\|Z\|_{2,0}}.\] Moreover, we have equality on the left-hand side if and only if $Z$ is $\rank(Z)$-row sparse. 
\end{lemma}
\begin{proof}
Assume $\rank(Z)=r$. Let $U=(u_1,\dots, u_N)^T$ be as in the previous lemma where $u_1,\dots, u_N$ are the rows of $U$. 
Notice that $\|u_1\|_2,\dots, \|u_N\|_2\leq 1$ and 
\[\|U\|^2_{2,2}=\|u_1\|_2^2+\cdots +\|u_N\|_2^2=r\]
therefore 
\[\owl(Z)=\|u_1\|_2+\cdots +\|u_N\|_2\geq r.\]
So $\owl(Z)$ takes its minimum value when $\owl(Z)=r$ and that happens only when for each $u_i$, $\|u_i\|^2_2=\|u_i\|_2$ hence $\|u_i\|_2=0$ or $1$. Consequently, $\owl(Z)$ takes its smallest value for $Z\in \R^{N\times K}$ on the $\rank(Z)$-row sparse matrices. 

To prove the upper bound, note that only $s$ of the    $u_1,\dots, u_N$ are different from $0$, thus from H\"older's inequality 
\[
  \owl(Z)=\|U\|_{2,1}\leq \sqrt{s}\cdot \|U\|_{2,2}=\sqrt{s}\sqrt{r}.
\qedhere
\] 
\end{proof}




\begin{proposition}
\label{cor:phi-lsc}
\(\owl(Z)\) is lower semi-continuous.
\end{proposition}
\begin{proof}
The proof of this proposition is a direct corollary of Proposition~\ref{convergence:Psi} and the fact that the $\Gamma$-limits are always lower semicontinuous \cite[Proposition 6.8]{dal2012introduction}.
\end{proof}



Let  $Z=U\Sigma V^T$ be  singular value decomposition of  $Z\in \R^{N\times K}$. It can be full, thin or compact version, that does not affect the following notation but we will use the compact SVD unless stated otherwise. 
For a function $w:[0,\infty)\to [0,\infty)$ with $w(0)=0$, define 
\[w(Z):=Uw(\Sigma)V^T\]
where the matrix function is applied element-wise to the matrix $\Sigma$
containing the singular values. This is known in the literature by the ``generalized'' or ``singular value'' matrix calculus (see, e.g., \cite{AnderssonCarlssonPerfekt2015}).

Define the family of functions
\[
w_\gamma(\sigma) = \frac{\sigma}{\sqrt{\gamma + (1-\gamma)\sigma^2}} \quad\text{for }\gamma \in (0,1],
\]
and the limit function for \(\gamma \to 0\) by
\[
w_0(\sigma) = \begin{cases}
0 &\text{for } \sigma = 0, \\
1 &\text{for } \sigma > 0.
\end{cases}
\]
With this notation, we can write
\[
w_0(Z)=Z (Z^T Z)^{\dagger/2},
\quad\text{and}\quad 
w_\gamma(Z)=Z (\gamma I + (1-\gamma)Z^TZ)^{-1/2}
\quad\text{for } 0<\gamma  \leq 1.
\]
Thus, we can also write
\[
\Psi_\gamma(Z) = \norm{w_\gamma(Z)}_{2,1} = \norm{U w_\gamma(\Sigma)}_{2,1}, \quad\text{for } 0\leq \gamma  \leq 1
\]
where the last equality follows from the orthogonality of \(V\).

\begin{proposition}
  \label{prop:relax_basic_props}
\begin{enumerate}
\item \label{relax_basic_props_1}
  \(\Psi_1(Z) = \ellone(Z)\) and \(\Psi_0(Z) = \owl(Z)\).
\item \label{relax_basic_props_2}
  \(\sqrt{\gamma}\Psi_\gamma(Z) \leq  \ellone(Z)\)
  and \(\sqrt{1-\gamma}\Psi_\gamma(Z) \leq \owl(Z) \leq \sqrt{K N}\)
\item \label{relax_basic_props_3}
 \(\Psi_\gamma\) is Lipschitz continuous for any \(\gamma \in (0,1]\):
  \begin{align*}
  \abs{\Psi_\gamma(Z) - \Psi_\gamma(\widehat{Z})}
  \leq \sqrt{N/\gamma}\, \norm{Z-\widehat{Z}}_\fro.
  \end{align*}
\item \label{state:loc_lipsch} When \(\widehat{Z}^T\widehat{Z}\) is invertible, \(\owl(Z)\) is local Lipschitz continuous near $\widehat{Z}$ with 
\[|\owl(Z)-\owl(\widehat{Z})|\leq \max\{\norm{(Z^TZ)^{-1/2}}_\op, \norm{(\widehat{Z}^T\widehat{Z})^{-1/2}}_\op\} \|Z-\widehat Z\|_\fro.\]
  \end{enumerate}
\end{proposition}
\begin{proof}
The proofs of~\ref{relax_basic_props_1} and~\ref{relax_basic_props_2} are straight-forward.
Property~\ref{relax_basic_props_3} follows from \(\Psi_\gamma(Z) = \norm{w_\gamma(Z)}_{2,1}\), Lipschitz continuity of the norm with constant one
and~\cite[Theorem~1.1]{AnderssonCarlssonPerfekt2015}
\[
\norm{w_\gamma(Z)}_{2,1} - \norm{w_\gamma(\widehat{Z})}_{2,1}
\leq \norm{w_\gamma(Z) - w_\gamma(\widehat{Z})}_{2,1}
\leq \sqrt{N} \norm{w_\gamma(Z) - w_\gamma(\widehat{Z})}_\fro
\leq  \sqrt{N/\gamma}\norm{Z-\widehat{Z}}_\fro
\]
due to \(\abs{w_\gamma(\sigma) - w_\gamma(\sigma')} \leq \gamma^{-1/2}\abs{\sigma-\sigma'}\) for all \(\sigma, \sigma'\).

Let \(Z\) and \(\widehat{Z}\) with singular values bounded from below by \(\min\{\sigma(Z),\sigma(\widehat{Z})\} = \tilde{\sigma} > 0\),
i.e.\ we set \(1/\tilde{\sigma} = \max\{\norm{(Z^TZ)^{-1/2}}_\op, \norm{(\widehat{Z}^T\widehat{Z})^{-1/2}}_\op\}\).
Define
\[
\tilde{w}_0(\sigma)
= \begin{cases}
\sigma/\tilde{\sigma} &\text{for } \sigma < \tilde{\sigma} \\
1 &\text{else.}
\end{cases}
\]
Then, it holds
\begin{align*}
    \abs{\owl(Z) - \owl(\widehat{Z})}&
= \abs{\norm{\tilde{w}_0(Z)}_{2,1} - \norm{\tilde{w}_0(\widehat{Z})}}_{2,1}\\&
\leq \norm{\tilde{w}_0(Z) - \tilde{w}_0(\widehat{Z})}_{2,1}\\&
\leq \sqrt{N}\norm{\tilde{w}_0(Z) - \tilde{w}_0(\widehat{Z})}_{2,2}\\&
\leq \sqrt{N}/\tilde{\sigma}\norm{Z - \widehat{Z}}_{2,2},
\end{align*}
where \(\tilde{w}_0(Z)\) is the singular value calculus of the function \(\tilde{w}_0(\sigma)\)
and we applied \cite[Theorem~1.1]{AnderssonCarlssonPerfekt2015}.
\end{proof}

\begin{proposition}
\label{convergence:Psi}
\(\Psi_\gamma\) \(\Gamma\)-converges to \(\Psi\) for \(\gamma \to 0\): For any \(Z^k \to Z\) and \(\gamma_k \to 0\) it holds
\[
\liminf_{k\to\infty} \Psi_{\gamma_k}(Z^k) \geq \owl(Z)
\]
and for any \(\gamma_k \to 0\) and \(Z \in \R^{N\times K}\) there exists a sequence \(Z^k \to Z\) such that
\[
\limsup_{k\to\infty} \Psi_{\gamma_k}(Z^k) \leq \owl(Z).
\]
\end{proposition}
\begin{proof}
To show the first condition, denote the
limit matrix \(Z\) and its full singular value decomposition by
\[
 Z = U \Sigma V,
\quad\text{where } V,\Sigma \in \R^{K\times K}, U \in \R^{N \times K}.
\]
Moreover, let a compatible SVD of \(Z_k\) be given by
\[
Z_k = U_k \Sigma_k V_k,
\quad\text{where } V_k,\Sigma_k \in \R^{K\times K}, U \in \R^{N \times K},
\]
such that \(\Sigma_k \to \Sigma\) for \(k\to\infty\) (e.g., by sorting the singular values by magnitude).
Then, we can write
\[
Z (Z^T Z)^{\dagger/2} = U w_0(\Sigma) V = w_0(Z) ,
\quad\text{and }
Z_k (\gamma_k I + (1-\gamma_k)Z_k^TZ_k)^{-1/2}
= U_k w_{\gamma_k}(\Sigma_k) V_k = w_{\gamma_k}(Z_k).
\]
Additionally let \(0 < \hat{\sigma} < \min\{\,\sigma_i\;|\;\sigma_i \in \operatorname{diag}(\Sigma), \sigma_i > 0\,\} \) be a lower bound on the smallest nonzero singular value of \(\Sigma\) and introduce the function
\[
\tilde{w}_\gamma(\sigma)
=
\begin{cases}
(\sigma/\hat{\sigma}) w_\gamma(\hat{\sigma}) &\text{for } \sigma\leq \hat{\sigma}\\
w_\gamma(\sigma) &\text{else.}
\end{cases}
\]
Clearly it holds \(w_\gamma \geq \tilde{w}_\gamma\) (by concavity of \(w_\gamma\)), \(\tilde{w}_\gamma\) is uniformly Lipschitz-continuous (independently of \(\gamma\)), and \(\tilde{w}_\gamma(\sigma) \to \tilde{w}_0(\sigma)\).
By elementary arguments, it holds for any \(\gamma\) and \(\Sigma\) that
\[
w_{\gamma}(\Sigma) \geq \tilde{w}_{\gamma}(\Sigma).
\]
Squaring this estimate leads for any \(v\) to the estimate
\[
\norm{v w_\gamma(\Sigma)}^2_2 = v^T w_\gamma(\Sigma)^2 v \geq  v^T \tilde{w}_\gamma(\Sigma)^2 v = \norm{v\tilde{w}_\gamma(\Sigma)}^2_2.
\]
Thus, it follows that
\[
\liminf_{k\to \infty} \Psi_{\gamma_k}(Z_k) 
= \liminf_{k\to \infty} \norm{U_k w_{\gamma_k}(\Sigma_k)}_{2,1} 
\geq \liminf_{k\to \infty} \norm{U_k \tilde{w}_{\gamma_k}(\Sigma_k)}_{2,1}
= \liminf_{k\to \infty}\norm{\tilde{w}_{\gamma_k}(Z_k)}_{2,1}.
\]
Using the fact that
\[
\norm{U_k \tilde{w}_{\gamma_k}(\Sigma_k)}_{2,1}
= \norm{U_k \tilde{w}_{\gamma_k}(\Sigma_k)V_k}_{2,1}
= \norm{\tilde{w}_{\gamma_k}(Z_k)}_{2,1}.
\]
To go to the limit we split the error into two terms:
\[
\tilde{w}_{\gamma_k}(Z_k)
= (\tilde{w}_{\gamma_k}(Z_k) - \tilde{w}_{0}(Z_k)) + (\tilde{w}_{0}(Z_k) - \tilde{w}_{0}(Z)) + \tilde{w}_{0}(Z).
\]
The first term can be estimated through a uniform estimate of \(\tilde{w}_{\gamma_k}(\cdot) - \tilde{w}_{0}(\cdot)\) on the compact interval \([0, \sup_{k,i} \sigma^k_i]\) and using the fact that \(U_k\) and \(V_k\) are uniformly bounded in any matrix norm. The second term can be estimated by the Lipschitz constant of \(\tilde{w}_{0}\) equal to \(1/\hat{\sigma}\) and the error of \(Z_k-Z\) using \cite[Theorem~1.1]{AnderssonCarlssonPerfekt2015}. Thus, indeed 
\[
\liminf_{k\to \infty} \Psi_{\gamma_k}(Z_k) 
\geq \liminf_{k\to \infty}\norm{\tilde{w}_{\gamma_k}(Z_k)}_{2,1}
= \norm{\tilde{w}_{0}(Z)}_{2,1} = \owl(Z),
\]
which shows the first assertion. For the second assertion, we can choose the recovery sequence \(Z_k = Z\) and obtain
\[
\Psi_{\gamma_k}(Z)
= \norm{U w_{\gamma_k}(\Sigma)}_{2,1}
\to \norm{U w_{0}(\Sigma)}_{2,1}
= \owl(Z),
\]
for \(k\to \infty\).
\end{proof}

In the following, we derive an analysis of the method based on splitting the nonconvex nonsmooth term \(\Psi_\gamma\) into a convex nonsmooth and a nonconvex quadratic remainder at a given fixed point \(\widehat{Z}\) (which can be thought of as an approximate local minimizer or current iterate \(Z_k\)).
For this, define the functions
\[
  M[Z] = \gamma I + (1-\gamma) Z^T Z
  \quad\text{and } W[Z] = M[Z]^{\dagger}
\]
and fix $\widehat{Z} \in \R^{N\times K}$ with associated $\widehat{M} = M[\widehat{Z}]$ and $\widehat{W} = W[\widehat{Z}]$ and assume that \(\widehat{M}\) is symmetric positive definite.
The last requirement is fulfilled for either \(\gamma > 0\) or if \(\widehat{Z}\) has full rank, i.e., \(\widehat{Z} \in \St(N,K)\). Then it also holds
\(\widehat{W} = \widehat{M}^{-1}\).
We recall the notation \(\|Z\|_{W,2} = \sqrt{\tr(ZWZ^T)} = \norm{ZW^{1/2}}_{\fro}.\)

\begin{theorem}\label{thm:first_order_approx}
Let $\widehat M = M[\widehat{Z}]$ be invertible.
Then, the functional $\Psi_\gamma(Z)$ is equal to a sum of convex part and  non-convex  quadratic remainder:
\begin{equation}
  \label{eq:expand_Psi}
  \Psi_\gamma(Z)
  = \norm{Z}_{W[Z],1}
  = \norm{Z}_{\widehat{W},1} + \langle \widehat{Z}\widehat{\Lambda}, Z-\widehat{Z} \rangle + R(Z,\widehat{Z}),
\end{equation}
where
\[
\widehat{\Lambda} 
= -\sum_{n=1}^N\frac{1-\gamma}{\norm{\widehat{z}_n}_{\widehat{W}}} \widehat{W} \widehat{z}_n \widehat{z}_n^T \widehat{W}.
\]
In fact, there exist constants \(\bar{L}\) and \(\bar{\Delta}\),
such that for all \(Z\) with \(\norm{Z-\widehat{Z}}_{\widehat{W},2} \leq \bar{\Delta}\),
\(M[Z]\) is also invertible, and there holds
\begin{equation}
  \label{eq:estimate_R}
\abs{R(Z,\widehat{Z})} 
\leq \frac{\bar{L}}{2} \norm{Z-\widehat{Z}}^2_{\widehat W, 2}.
\end{equation}
\end{theorem}

\begin{remark}
\label{rem:Lambda}
We note that the terms in the definition of \(\widehat{\Lambda}\) with \(\widehat{z}_n = 0\) are not defined properly. They are evaluated as zero by convention, which is consistent with the limit \(\widehat{z}_n \to 0\) of each expression under the sum.
  Note also that \(-\widehat{\Lambda}\) is a symmetric nonnegative matrix with
\(-\tr(\widehat{W}^{-1/2}\widehat{\Lambda}\widehat{W}^{-1/2}) = (1-\gamma) \norm{\widehat{Z}}_{\widehat{W},1}\).
Thus, the linear term in the expansion~\eqref{eq:expand_Psi} can be estimated as
\[
  \langle \widehat{Z}\widehat{\Lambda}, Z-\widehat{Z} \rangle
  \leq (1-\gamma) \norm{\widehat{Z}}_{\widehat{W},1} \|\widehat{Z}\widehat{W}^{1/2}\|_\op \norm{Z-\widehat{Z}}_{\widehat{W}, 2}
  \leq (1-\gamma)^{1/2} \norm{\widehat{Z}}_{\widehat{W},1} \norm{Z-\widehat{Z}}_{\widehat{W}, 2}.
\]
\end{remark}

To prove this result, we first derive some continuity and differentiability properties of \(W\).
\begin{proposition}
  \label{prop:matrix_fun_bound}
For every \(Z\) and \(W = W[Z]\) there holds \(\|W^{1/2}Z^T\|_\op = \|ZW^{1/2}\|_\op \leq (1-\gamma)^{-1/2}\).
\end{proposition}
\begin{proof}
  With the singular value calculus we have
  \(ZW^{1/2} = w_\gamma(Z) = U w_\gamma(\Sigma) V\)
  and \(w_\gamma(\cdot) \leq (1-\gamma)^{-1/2}\).
\end{proof}
\begin{proposition}
  \label{prop:W_perturb}
  For \(W = W[Z]\), \(\widehat{W} = W[\widehat{Z}]\) such that both are invertible
  there holds
\begin{align}
  \abs{z^T (W - \widehat{W}) z }
  &\leq C_W  \norm{z}^2_{\widehat{W}}\norm{(Z-\widehat{Z})\widehat{W}^{1/2}}_{\op}
    \quad\text{for all } z \in \R^K,
    \label{eq:W_diff_estimate}
\end{align}
where
\(C_W = (1-\gamma)^{1/2} \norm{\widehat{W}^{-1}W}_\op(2 + \norm{(Z-\widehat{Z})\widehat{W}^{1/2}}_{\op})\).
Moreover with the perturbation of \(W\) defined by
\[
  D_W = -2 (1-\gamma) \widehat{W}\widehat{Z}^T(Z-\widehat{Z})\widehat{W}
\]
and
\(
C'_W = (1-\gamma)\left(1 + \norm{\widehat{W}^{-1}W}_\op(2 + \norm{(Z-\widehat{Z})\widehat{W}^{1/2}}_{\op})^2\right)
\)
there holds
\begin{align}
  \abs{z^T (W - \widehat{W} - D_W) z }
  &\leq C'_W  \norm{z}^2_{\widehat{W}}\norm{(Z-\widehat{Z})\widehat{W}^{1/2}}^2_{\op}
    \quad\text{for all } z \in \R^K.
    \label{eq:W_deriv_estimate}
\end{align}
\end{proposition}
\begin{proof}
Denote  \(M = M[Z] = \gamma I + (1-\gamma)Z^T Z\). We have that
\begin{align}
  W - \widehat W&=  M^{-1} - \widehat M^{-1} = \widehat{M}^{-1}(\widehat{M} - M)M^{-1} = \widehat{W}(\widehat{M} - M)W \label{eq:expand_W1}\\
                &= - \widehat{W}(M - \widehat{M}) \widehat{W}
                   + \widehat{W}(M - \widehat{M})\widehat{W}(M - \widehat{M})W. \label{eq:expand_W2}
\end{align}
Turning now to the terms with differences of \(M\), note that
\begin{align}
   M - \widehat{M} &= (1-\gamma) (Z^TZ-\widehat{Z}^T\widehat{Z}) \notag \\
   &=(1-\gamma) (Z^T(Z-\widehat{Z}) + (Z-\widehat{Z})^T\widehat{Z}) \notag \\
                   &= (1-\gamma) (\widehat{Z}^T(Z-\widehat{Z})+(Z-\widehat{Z})^T\widehat{Z} + (Z-\widehat{Z})^T(Z-\widehat{Z})).\label{eq:expand_M2}
\end{align}
Thus we have with Proposition~\ref{prop:matrix_fun_bound} that
\begin{equation}
\label{eq:M_diff_estimate}
\begin{aligned}
  \|(M - \widehat{M})\widehat{W}\|_{\op}
&\leq (1-\gamma)(\|\widehat{Z}^T(Z-\widehat{Z})\widehat W\|_\op+\|(Z-\widehat{Z})^T\widehat{Z}\widehat W\|_\op+\|(Z-\widehat{Z})^T(Z-\widehat{Z})\widehat W\|_\op)\\
&\leq 
(1-\gamma)(2\|\widehat W^{1/2}\widehat{Z}^T\|_\op+\norm{(Z-\widehat{Z}) \widehat{W}^{1/2}}_{\op})\norm{(Z-\widehat{Z}) \widehat{W}^{1/2}}_{\op}\\
&\leq 
(2(1-\gamma)^{1/2}+(1-\gamma)\norm{(Z-\widehat{Z}) \widehat{W}^{1/2}}_{\op})\norm{(Z-\widehat{Z}) \widehat{W}^{1/2}}_{\op}.
\end{aligned}
\end{equation}
Combining this with~\eqref{eq:expand_W1}
 we already obtain~\eqref{eq:W_diff_estimate}
due to
\begin{align*}
  \abs{z^T (W - \widehat{W}) z }
  &=  \abs{z^T \widehat{W}(M - \widehat{M})W z }
  \leq \norm{z}^2_{\widehat{W}} \norm{\widehat{W}^{1/2}(M - \widehat{M})W\widehat{W}^{-1/2}}_\op
  \\
  &\leq \norm{z}^2_{\widehat{W}} \norm{(M - \widehat{M})\widehat{W}}_\op \norm{\widehat{W}^{-1}W}_\op.
\end{align*}

It remains to derive~\eqref{eq:W_deriv_estimate}.
We start by noticing that
\[
  z^T D_W z = (1/2) z^T (D_W + D_W^T) z
  = - (1-\gamma) z^T\left(\widehat{Z}^T(Z-\widehat{Z}) + (Z-\widehat{Z})^T\widehat{Z} \right)z 
\]
Using this together using~\eqref{eq:expand_W2} and~\eqref{eq:expand_M2} we obtain
\[
z^T (W - \widehat{W} - D_W) z 
= - (1-\gamma) \left(z^T\widehat{W}(Z-\widehat{Z})^T(Z-\widehat{Z}) \widehat{W}z
                   + z^T\widehat{W}(M - \widehat{M})\widehat{W}(M - \widehat{M})Wz \right)
\]
Taking the absolute value of that expression and using the triangle inequality, the first term is estimated by
\[
  \abs{z^T\widehat{W}(Z-\widehat{Z})^T(Z-\widehat{Z}) \widehat{W}z}
  \leq  \norm{z}^2_{\widehat W} \norm{(Z-\widehat{Z}) \widehat{W}^{1/2}}_{\op}^2
\]
Concerning the second term, note that 
\begin{align*}
    \abs{z^T\widehat{W}(M - \widehat{M})\widehat{W}(M - \widehat{M})Wz} &
    \leq \norm{z}^2_{\widehat W}  \|\widehat{W}^{1/2}(M - \widehat{M})\widehat{W}(M - \widehat{M})W\widehat{W}^{-1/2}\|_\op\notag\\
           &
    = \norm{z}^2_{\widehat W}  \|(M - \widehat{M})\widehat{W}(M - \widehat{M})\widehat W \widehat W^{-1}W\|_\op\notag\\
      &
    \leq \norm{z}^2_{\widehat W} \| \widehat W^{-1}W\|_\op \|(M - \widehat{M})\widehat{W}\|^2_\op. 
\end{align*}
and together with~\eqref{eq:M_diff_estimate} we obtain the second result~\eqref{eq:W_deriv_estimate}.
\end{proof}

\begin{proposition}\label{prop:rank_one_fro}
The mapping
\[
\R^K \ni z \mapsto \frac{z z^T}{\norm{z}} \in \R^{K\times K}
\]
is Lipschitz-continuous with constant three in the Frobenius norm.
\end{proposition}
\begin{proof}
  First, we consider the case where \(z = 0\). In that case
  \[
    \norm*{\frac{z' z'^T}{\norm{z'}} - \frac{z z^T}{\norm{z}}}_\fro
    = \norm*{\frac{z' z'^T}{\norm{z'}}}_\fro = \norm{z'}_2 = \norm{z'-z}_2
  \]
  In the nonzero case
  \[
    \norm*{\frac{z' z'^T}{\norm{z'}} - \frac{z z^T}{\norm{z}}}_\fro
    = \frac{1}{\norm{z'}\norm{z}}
    \norm*{(z' - z) z'^T\norm{z} + z z'^T (\norm{z} - \norm{z'}) + z (z'^T - z^T)\norm{z'}}_\fro
    \leq 3 \norm{z'-z}_2 \qedhere
  \]
\end{proof}

We are ready to prove the main result.
\begin{proof}[Proof of Theorem~\ref{thm:first_order_approx}]
Start with the identity for \(a,\widehat{a} > 0\)
\[
  \sqrt{a} - \sqrt{\widehat{a}}
  = \frac{2\sqrt{a}\sqrt{\widehat{a}} - 2\widehat{a}}{2\sqrt{\widehat{a}}}
  = \frac{a - \widehat{a} - \big(\sqrt{a} - \sqrt{\widehat{a}}\big)^2}{2\sqrt{\widehat{a}}}
  = \frac{a - \widehat{a}}{2\sqrt{\widehat{a}}}
  - \frac{(a - \widehat{a})^2}{2\sqrt{\widehat{a}}\big(\sqrt{a} + \sqrt{\widehat{a}}\big)^2}.
\]
Thus, denote $W = W[Z]$ and $\widehat{W} = W[\widehat{Z}]$, where both are invertible, and using
\(a= z_n^T W z_n\) and \(\widehat{a} = z_n^T \widehat{W} z_n\) for \(z_n \neq 0\) we have
\begin{align}
  \|Z\|_{W,1} - \|Z\|_{\widehat W,1}
  &= \sum_{n\colon z_n\neq 0} \sqrt{z_n^T W z_n}
   - \sqrt{z_n^T \widehat W z_n}\notag\\
       &=
         \sum_{n\colon z_n\neq 0} \frac{z_n^T (W-\widehat W) z_n}{2\|z_n\|_{\widehat W}}
         -\frac{(z_n^T (W - \widehat W) z_n)^2}{2\|z_n\|_{\widehat W}(\|z_n\|_{ W} +\|z_n\|_{\widehat W})^2}.
         \label{eq:expand_sqrt}
\end{align}
We define the intermediate value
\[
  \Lambda = -\sum_{n=1}^N\frac{1-\gamma}{\norm{z_n}_{\widehat{W}}} \widehat{W} z_n z_n^T \widehat{W}
\]
and note that by cyclic permutation
\[
\langle  \widehat{Z}\Lambda , (Z - \widehat{Z}) \rangle
=\tr\left(\Lambda \widehat{Z}^T (Z - \widehat{Z})\right)
= - \sum_{n=1}^N \frac{2(1-\gamma)z_n^T\widehat{W}\widehat{Z}^T(Z-\widehat{Z})\widehat{W}z_n}{2\|z_n\|_{\widehat W}}
= \sum_{n=1}^N \frac{z_n^TD_W z_n}{2\|z_n\|_{\widehat W}},\\
\]
where \(D_W\) is as in Proposition~\ref{prop:W_perturb}. Subtracting this from both sides of~\eqref{eq:expand_sqrt}, it follows
\[
  \|Z\|_{W,1} - \|Z\|_{\widehat W,1} - \langle  \widehat{Z}\Lambda , (Z - \widehat{Z}) \rangle
  = R_0(Z,\widehat{Z}).
\]
where
\[
  R_0(Z,\widehat{Z}) =
\sum_{n\colon z_n\neq 0} \frac{z_n^T (W - \widehat{W}- D_W) z_n}{2\|z_n\|_{\widehat W}}
         -\frac{(z_n^T (W - \widehat W) z_n)^2}{2\|z_n\|_{\widehat W}(\|z_n\|_{ W} +\|z_n\|_{\widehat W})^2}
\]
By the triangle inequality and Proposition~\ref{prop:W_perturb} we obtain
\begin{align*}
|R_0(Z,\widehat Z)|
  &\leq
    \sum_{n\colon z_n\neq 0} \frac{\norm{z_n}_{\widehat{W}}}{2}
    \left(C'_W + C^2_W\right)\norm{(Z-\widehat{Z})\widehat{W}^{1/2}}_{\op} 
\leq \frac{L_1}{2}  \norm{Z-\widehat{Z}}^2_{\widehat{W}, 2}
\end{align*}
where 
\[
  L_1 = \left(C'_W + C^2_W\right)\|Z\|_{\widehat{W}, 1}.
\]
To obtain the final result, it remains to estimate
the difference of \(\Lambda\) and \(\widehat{\Lambda}\).
We start with
\[
  \widehat{W}^{-1/2}(\Lambda-\widehat{\Lambda})\widehat{W}^{-1/2}
  =
-(1-\gamma)\sum_{n=1}^N\frac{1}{\norm{z_n}_{\widehat{W}}}  \widehat{W}^{1/2}z_n z_n^T\widehat{W}^{1/2} -
\frac{1}{\norm{\widehat{z}_n}_{\widehat{W}}}  \widehat{W}^{1/2}\widehat{z}_n^T \widehat{z}_n\widehat{W}^{1/2}
\]
By the triangle inequality and Proposition~\ref{prop:rank_one_fro} we have
\[
  \norm{(\Lambda-\widehat{\Lambda})\widehat{W}^{-1}}_\fro =
  \norm{\widehat{W}^{-1/2}(\Lambda-\widehat{\Lambda})\widehat{W}^{-1/2}}_\fro
  \leq 3(1-\gamma) \sum_{n=1}^N \norm{z_n - \widehat{z}_n}_{\widehat{W}} 
\]
and thus
\begin{align*}
\abs{\langle  \widehat{Z}(\Lambda - \widehat{\Lambda}), (Z - \widehat{Z}) \rangle}
=
  &
  \abs*{\tr\left((\Lambda-\widehat{\Lambda}) \widehat{Z}^T (Z - \widehat{Z})\right)}\\
 \leq \;&
\norm{(\Lambda-\widehat{\Lambda})\widehat{W}^{-1}}_\fro
\|\widehat{W} \widehat{Z}^T(Z-\widehat{Z})\|_\fro\\
\leq \;& 3(1-\gamma) \norm{Z-\widehat{Z}}_{\widehat{W},1}
\norm{\widehat{W}^{1/2} \widehat{Z}^T}_{\op}\norm{Z-\widehat{Z}}_{\widehat{W},2}\\
\leq \;& 3\sqrt{N}(1-\gamma)^{1/2}\norm{Z-\widehat{Z}}^2_{\widehat{W},2},
\end{align*}
where we have used again Proposition~\ref{prop:matrix_fun_bound} and the inequality \(\norm{\cdot}_{\widehat{W},1}\leq\sqrt{N}\norm{\cdot}_{\widehat{W},2}\).

The previous estimates combined lead to the desired expansion of \(\Psi_\gamma\)~\eqref{eq:expand_Psi}
with the estimate of the remainder~\eqref{eq:estimate_R} with the constant
\[
  L = 6\sqrt{N}(1-\gamma)^{1/2} + L_1.
\]
  The constant \(L\) depends on 
  \(\norm{(Z-\widehat{Z})\widehat{W}^{1/2}}_{\op}\leq \norm{Z-\widehat{Z}}_{\widehat{W},2}\)
  and \(\norm{\widehat{W}^{-1}W}_\op\) through Proposition~\ref{prop:W_perturb}, and on \(\norm{Z}_{\widehat{W}, 1}\).
  Moreover, the latter two quantities can also be estimated in terms of \(\norm{Z-\widehat{Z}}_{\widehat{W},2}\).
  First, we have
  \[
    \norm{\widehat{W}^{-1}W}_\op
    = \norm{I + (\widehat{M} - M)\widehat{W}\widehat{W}^{-1}W}_\op
    \leq 1 + \norm{(\widehat{M} - M)\widehat{W}}_\op\norm{\widehat{W}^{-1}W}_\op
  \]
  and substituting \(e_M = \norm{(\widehat{M} - M)\widehat{W}}_\op\) and rearranging we obtain
  \((1 - e_M) \norm{\widehat{W}^{-1}W}_\op \leq 1\), i.e.,
  \(\norm{\widehat{W}^{-1}W}_\op \leq 1/(1-e_M) = 1 + e_M/(1-e_M)\).
  Now, we can apply~\eqref{eq:M_diff_estimate}:
  \(e_M \leq (2 + \norm{(Z-\widehat{Z})\widehat{W}^{1/2}}_\op)\norm{(Z-\widehat{Z})\widehat{W}^{1/2}}_\op\).
  Second, we have
  \[
    \norm{Z}_{\widehat{W}, 1}
    \leq \norm{\widehat{Z}}_{\widehat{W}, 1} + \norm{Z-\widehat{Z}}_{\widehat{W}, 1}
  \]
  Thus,
  taking into account the inequalities \(\norm{\cdot}_{\widehat{W},1}
  \leq\sqrt{N}\norm{\cdot}_{\widehat{W},2} \leq \sqrt{NK}\norm{\cdot\, \widehat{W}}_{\op}\)
  and Proposition~\ref{prop:matrix_fun_bound},
  it can be seen that for \(\norm{Z-\widehat{Z}}_{\widehat{W},2}\) smaller than a universal constant \(\bar{\Delta}\),
  the constant \(L\) can be bounded by some universal constant \(\bar{L}\).
\end{proof}

Theorem~\ref{thm:first_order_approx} allows to explicitly compute the Clarke subdifferential with the convex subdifferential of the model function for \(\Psi_\gamma\) defined as motivated by the previous theorem
\begin{equation}
\label{eq:model_Psi_app}
  m^{\Psi}_{\widehat{Z}}(Z) =
\norm{Z}_{\widehat{W},1}
+ \langle \widehat{Z}\widehat{\Lambda}, Z - \widehat{Z}\rangle.
\end{equation}
\begin{theorem}\label{thm:clark_subdifferential}
For all \(\widehat Z\) with $\widehat M = M[\widehat{Z}]$ invertible we have
\[
\partial_C \Psi_\gamma(\widehat Z) = \partial m^{\Psi}_{\widehat Z}(\widehat Z).
\]
\end{theorem}
\begin{proof}
Assume first that \(Z\) is given with \(z_n \neq 0\) for all \(n\). Then, \(\Psi_\gamma\) is Frechet differentiable with
\[
  [\nabla \Psi_\gamma(Z)]_n
  = \frac{z_n}{\norm{z_n}_{W}} W - (1-\gamma)z_n \sum_{i=1}^N \frac{1}{\norm{z_i}_{W}} W z_i^T z_i W
  \quad\text{where } W = W[Z],
\]
which is obtained using the differentiability of \(\norm{Z}_{W,1}\) for fixed \(W\) and the chain rule.
On the other hand \(m^\Psi_{\widehat{Z}}(Z)\) is also differentiable under this assumption with
\[
[\nabla m^\Psi_{\widehat{Z}}(Z)]_n 
= \frac{z_n}{\norm{z_n}_{\widehat{W}}} \widehat{W}
- (1-\gamma)\widehat{z}_n\sum_{i=1}^N \frac{1}{\norm{\widehat{z}_i}_{\widehat{W}}} \widehat{W} \widehat{z}_i^T \widehat{z}_i \widehat{W}
\]
Furthermore for any sequence \(Z_k \to \widehat{Z}\) with the above property it holds
\begin{equation}
\label{eq:limit_differentials}
\lim_k \nabla m^{\Psi}_{\widehat{Z}}(Z_k) = \lim_k \nabla \Psi_\gamma(Z_k),
\end{equation}
assuming that either of the limits exists, i.e., \(z_n/\norm{z_n}_{\widehat{W}} \to s_n\) for all \(n\).
Now, by the characterization of the Clarke subdifferential (\cite[Theorem~10.27]{clarke2013functional}) we know that
\[
  \partial_C\Psi_\gamma(\widehat{Z})
  = \operatorname{conv}\operatorname*{Limits}_{Z\to\widehat{Z}: z_n \neq 0\, \forall n}  \nabla\Psi_\gamma(Z)
\]
and also that
\[
  \partial m^{\Psi}_{\widehat{Z}}(\widehat{Z})
  = \partial_C m^{\Psi}_{\widehat{Z}}(\widehat{Z})
  = \operatorname{conv}\operatorname*{Limits}_{Z\to\widehat{Z}: z_n \neq 0\, \forall n}  \nabla m^{\Psi}_{\widehat{Z}}(Z).
\]
Thus, by~\eqref{eq:limit_differentials}, both differentials agree.
\end{proof}

\end{document}